\documentclass[12pt,reqno]{amsart}

\usepackage[margin=1in]{geometry}
\usepackage[T1]{fontenc}
\IfFileExists{libertine.sty}{\usepackage[tt=false]{libertine}}{\usepackage{lmodern}}
\usepackage{amssymb,mathtools,microtype}
\IfFileExists{newpxmath.sty}{\usepackage[varbb]{newpxmath}}{}
\IfFileExists{mathabx.sty}{%
  \let\savedbigtimes\bigtimes
  \let\bigtimes\relax
  \usepackage{mathabx}
  \let\bigtimes\savedbigtimes
}{}
\usepackage{graphicx,enumerate,bm,verbatim}
\usepackage[dvipsnames]{xcolor}
\usepackage{hyperref}
\usepackage[capitalize,nameinlink]{cleveref}
\usepackage{appendix}
\hypersetup{
  colorlinks=true,
  pdfpagemode=UseNone,
  citecolor=OliveGreen,
  linkcolor=NavyBlue,
  urlcolor=black,
}

\theoremstyle{plain}
\newtheorem{theorem}{Theorem}[section]
\newtheorem{proposition}[theorem]{Proposition}

\newtheorem{lemma}[theorem]{Lemma}

\theoremstyle{definition}

\theoremstyle{remark}
\newtheorem{remark}[theorem]{Remark}

\crefname{lemma}{Lemma}{Lemmas}
\crefname{theorem}{Theorem}{Theorems}
\crefname{definition}{Definition}{Definitions}
\crefname{proposition}{Proposition}{Propositions}
\crefname{appsec}{Appendix}{Appendices}

\newcommand{\dif}{\,\mathrm{d}}
\newcommand{\E}{\mathbb{E}}
\newcommand{\Var}{\mathrm{Var}}
\newcommand{\Cov}{\mathrm{Cov}}

\newcommand{\norm}[1]{\left\lVert #1 \right\rVert}
\newcommand{\ip}[2]{\left\langle #1 , #2 \right\rangle}
\newcommand{\grad}{\nabla}
\newcommand{\hessian}{\nabla^2}

\newcommand{\diag}{\mathrm{diag}}
\newcommand{\sgn}{\mathrm{sgn}}
\newcommand{\eps}{\varepsilon}
\renewcommand{\epsilon}{\varepsilon}

\newcommand{\R}{\mathbb{R}}

\newcommand{\beq}{\begin{equation}}
\newcommand{\eeq}{\end{equation}}

\newcommand{\cE}{\mathcal E}

\newcommand{\we}{c}

\numberwithin{equation}{section}
\allowdisplaybreaks[2]
\title[A dimension-free bound for isoperimetry of unconditional log-concave measures]{A dimension-free bound for the isoperimetric constant\\ of unconditional log-concave measures}

\begin{document}
\author
[D. Mikulincer, I. Zadik]{Dan Mikulincer$^{\circ}$ and  Ilias Zadik$^{\dagger}$}

\thanks{\thanks{\raggedright$^\circ$Department of Mathematics, University of Washington;
$^\dagger$Department of Statistics and Data Science, Yale University.
Email: \texttt{danmiku@gmail.com, ilias.zadik@yale.edu}}}

\date{\today}

\begin{abstract}
We prove a dimension-free Poincar\'e inequality for isotropic
unconditional log-concave probability measures, establishing the
Kannan--Lov\'asz--Simonovits (KLS) conjecture in the unconditional setting.
The proof combines a new application of Dunkl operator theory combined with an appropriate transformation of the log-concave measure. Specifically, the proof follows by a direct spectral analysis of a new Dunkl–Langevin operator adapted to the transformed measure. The core ideas of the proof were generated by AI generative tools through interactions with the authors over several weeks. The authors refined and formalized the argument.
\end{abstract}

\maketitle

%\tableofcontents

\section{Introduction} For a probability measure $\mu$ on $\R^n$, we say it satisfies a Poincar\'e inequality, if there exists a constant $C>0$, such that for any locally Lipschitz test function $f:\R^n\to \R$,
\begin{equation} \label{eq:poincare}
    \mathrm{Var}_\mu(f) \leq C\E_{\mu}\left[|\nabla f|^2\right].
\end{equation}
When \eqref{eq:poincare} holds, we denote the best choice of the constant $C$ on the right-hand side by $C_{\mathrm{P}}(\mu)$, and call it the Poincar\'e constant of $\mu.$
Our focus in this work is on bounding the Poincar\'e constant of unconditional log-concave measures.

To give the necessary definitions, we call a function $V:\R^n\to \R\cup\{\infty\}$ unconditional if $V$ is invariant under the reflection of each coordinate. That is, for every $j \in[n]$, 
$$V(x_1,\dots, x_j,\dots,x_n) = V(x_1,\dots, -x_j,\dots,x_n).$$ Equivalently, for every $x \in \R^n$, $V(x_1,\dots,x_n) = V(|x_1|,\dots,|x_n|)$. Now, suppose that $\mu$ has a density of the form $e^{-V}$. We say that $\mu$ is unconditional if $V$ is an unconditional function. If $V$ is also convex, then we call $\mu$ log-concave. We also call $\mu$ isotropic if it is centered and its covariance is the identity, i.e.
$$\E_{\mu}[X] = 0 \quad \text{ and } \quad \E_\mu[X\otimes X] = \mathrm{I}.$$
With these definitions, our main result is the following.
\begin{theorem}\label{thm:main}
There exists a universal constant $C>0$, independent of the dimension $n$, such that for every isotropic, unconditional, and log-concave measure $\mu$ on $\R^n$,
\begin{equation}\label{eq:main}
 C_\mathrm{P}(\mu)\leq C.
\end{equation}
\end{theorem}
Let us make some remarks concerning Theorem \ref{thm:main}. For a measure $\mu$, define $\psi_\mu$, its isoperimetric Cheeger constant by
$$\frac{1}{\psi_\mu} = \inf_{A\subset\mathbb{R}^n}
\frac{\mu^+(\partial A)}
     {\min\{\mu(A),1-\mu(A)\}},$$
where the infimum is taken over all measurable $A$ with smooth boundary with $\mu(A) \in (0,1)$, and the boundary measure is defined by
$$\mu^+(\partial A) = \liminf\limits_{\eps\to 0}\frac{\mu\left(\{x\in \R^n| \mathrm{dist}(x,A)\leq \eps\}\right)-\mu(A)}{\eps}.$$
When $\mu$ is log-concave, it is known that $\psi_\mu^2\simeq C_{\mathrm{P}}(\mu)$, see \cite{milman2009role}. In fact, by \cite[Corollary 1.2]{deponti2021},
$$\frac{1}{4}
\leq \frac{\psi_\mu^2}{C_\mathrm{P}(\mu)}
\leq \pi.$$
Thus, Theorem \ref{thm:main} can also be stated as
$$\psi_\mu\leq C,$$
for every isotropic, unconditional, and log-concave measure $\mu$. The same statement, without the unconditionality assumption, is precisely the object of the celebrated Kannan---Lov\'asz---Simonovits (KLS) conjecture, which aims to understand the isoperimetric constant of convex sets and, by extension, of general log-concave measures. We refer the reader to \cite{lee2019kannan,alonso2020kannan} for further background on the conjecture and its relation to other fields.

On the topic of Theorem \ref{thm:main} we mention that for any unconditional isotropic log-concave measure $\mu$, in \cite{klartag2009berry}, Klartag gave a bound of the form $C_\mathrm{P}(\mu) = O(\log^2(n))$. Subsequently, \cite{barthe2020spectral} followed up with a more detailed investigation of Poincar\'e inequalities in the presence of symmetries. We also remark that Klartag's $\log^2(n)$ bound was later superseded by general bounds towards the KLS conjecture, see \cite{klartag2025isoperimetric} for a general overview of the involved methods.

We finally remark that we made no attempt to optimize the universal constant $C$ in Theorem~\ref{thm:main}, and a naive check gives $C$ approximately $10^{12}$. Yet, preliminary AI-assisted calculations suggest that the proof may be sharpened to obtain a significantly smaller constant $C,$ potentially $C \leq 8$. Note that an easy calculation for the isotropic product Laplace distribution requires $C \geq 2.$ We leave the optimization of the constant to future work.

\subsection{Ideas of the proof} \label{sec:ideas}
The main principle underlying the proof of Theorem \ref{thm:main} is to capitalize on the symmetries afforded by the unconditionality condition. To explain how to exploit these symmetries for a measure $\mu = e^{-V}dx$ we introduce its associated Langevin operator $L$, acting on functions $f:\R^n\to\R$ by
$$Lf = -\Delta f +\langle \nabla V, \nabla f\rangle.$$
When $\mu$ is log-concave, it is well-known that $L$ is a self-adjoint non-negative operator in $L^2(\mu)$ and that through integration by parts, for every two test functions $f,g$,
$$\E_\mu\left[fLg\right] = \E_\mu\left[\langle \nabla f,\nabla g\rangle \right] = \E_\mu\left[gLf\right],$$
see \cite{klartag2025isoperimetric}. Furthermore, it holds that $\mathrm{ker}(L) = \{\textnormal{constant functions}\}$, and so from the definition in \eqref{eq:poincare} we can see that
$$\frac{1}{C_{\mathrm{P}}(\mu)} = \inf\limits_{g:\mathrm{Var}_\mu(g)>0} \frac{\E_\mu\left[|\nabla g|^2\right]}{\mathrm{Var}_\mu(g)} =\inf\limits_{g:\mathrm{Var}_\mu(g)= 1} \E_\mu\left[gLg\right].$$
In other words, since $L$ is self-adjoint, $\frac{1}{C_{\mathrm{P}}(\mu)}$ is the spectral gap of $L$ and corresponds to the first non-trivial eigenvalue of $L$. Thus, if $f$ is a normalized eigenfunction of $L$ such that $\frac{1}{C_{\mathrm{P}}(\mu)} = \E_\mu\left[fLf\right],$ one can study the Poincar\'e constant through understanding the eigenfunction $f$. 

Our starting point is the work \cite{barthe2020spectral} who identified potential symmetries in the eigenfunction when the measure is unconditional. Specifically, they show that, when the measure is log-concave and satisfies some regularity assumptions, such an eigenfunction $f$ exists and that it can be chosen to be \emph{odd} in at least one coordinate. Moreover, their method also shows that one can choose $f$ to be \emph{even} in all other coordinates. The question is now how to use this information to establish bounds on $C_{\mathrm{P}}(\mu)?$ More precisely, how can we bound $\E_\mu\left[fLf\right]$,  from below, in the presence of symmetries? 

Our proof contains two main components towards answering the above question. The first component is to change the definition of $L$, and construct a new operator. The new operator will have extra terms, beyond the ones in $L$, which can be shown to be positive when a function $f$ has the symmetries uncovered in \cite{barthe2020spectral}. Let us explain some of the intuition behind the construction. For $j \in [n]$, consider $R_j$ the reflection operator so that
$$R_jf(x_1,\dots,x_j,\dots,x_n) = f(x_1,\dots,-x_j,\dots,x_n).$$
Using these reflections we define a new operator
\begin{equation} \label{eq:Aheuristic}
    Af(x) = Lf(x) + \frac{1}{2}\sum\limits_{j}\frac{1}{2x_j^2}(f-R_jf).
\end{equation}

We first note that if $f$ is even in coordinate $j$, then $(f-R_jf) = 0$, while if $f$ is odd there then, $(f-R_jf) = 2f$. Thus, for any function with the above described symmetries
\begin{equation}\label{eq:AtoL}
    \E_\mu[fAf] = \E_\mu[fLf] + \frac{1}{2}\E_{\mu}\left[\frac{f^2}{X^2_i}\right], 
\end{equation}
where $i\in[n]$ is the unique coordinate in which $f$ is odd. In essence, we have gained a strictly positive term to help with proving a lower bound. Of course, there are a few immediate obstacles to this strategy.
\begin{itemize}
    \item A lower bound on the right-hand side of \eqref{eq:AtoL} can, at best, establish a spectral gap for the operator $A$, not $L$. This is not very useful unless there is a bound in the reverse direction $\E_\mu[fAf] \lesssim \E_\mu[fLf].$
    \item More importantly, the operator $A$ is not well-behaved on $L^2(\mu)$. For example, if $f$ is linear, then $Af \notin L^2(\mu)$. Thus, the discussion about eigenfunctions with specific symmetries becomes somewhat difficult.
\end{itemize}
Addressing these two points is the second component in our proof. Before explaining this component, let us remark that, in light of \eqref{eq:AtoL} showing $\E_\mu[fAf] \lesssim \E_\mu[fLf]$, at least for our desired eigenfunction, reduces to an inequality of the form
\begin{equation} \label{eq:hardy}
    \E_{\mu}\left[\frac{f^2}{X_i^2}\right] \lesssim \E_{\mu}\left[|\nabla f|^2\right].
\end{equation}
Inequalities such as \eqref{eq:hardy} are known in the literature as Hardy-type inequalities \cite{muchenhoupt1972hardy}. Most relevant to our problem, the Hardy-type inequalities are known to hold in the presence of a weight that vanishes near $\{x_j = 0\},$ which leaves room to change the measure $\mu$ and adapt it to our needs.

\IfFileExists{eta_s_geometry_updated.pdf}{%
\begin{figure}[tbp]
  \centering
  \includegraphics[width=\linewidth]{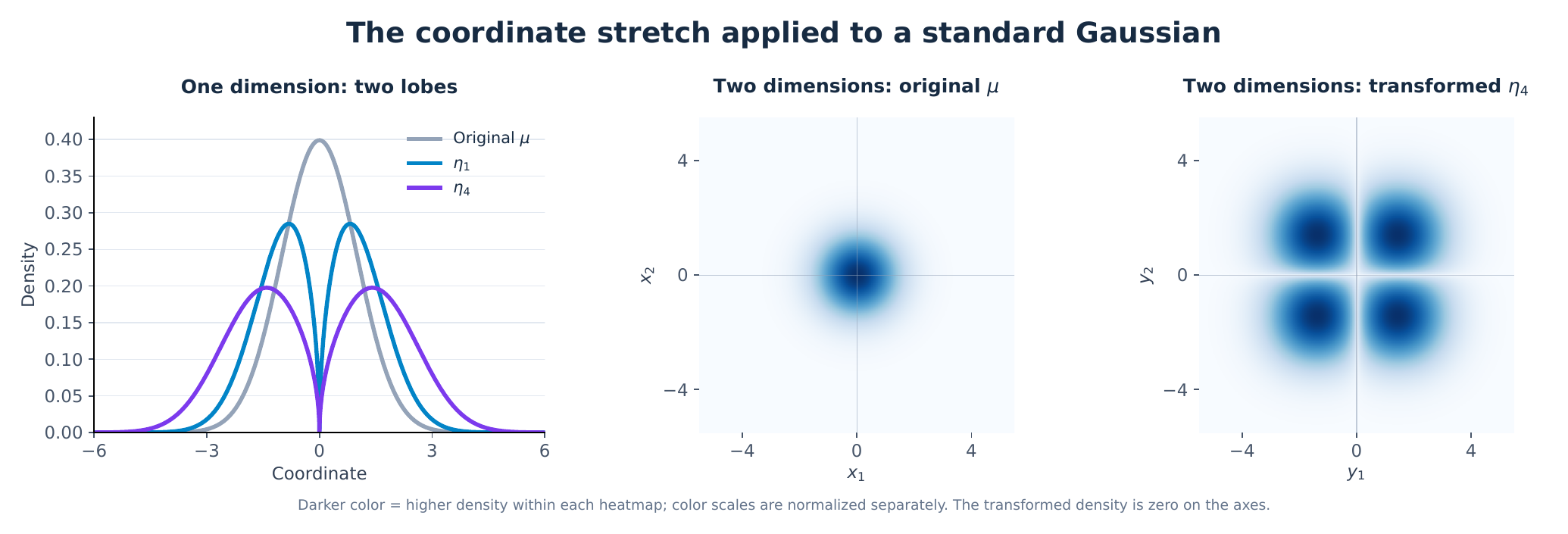}
  \caption{The $\eta$ transform for the standard Gaussian measure $\mu$ in one and two dimensions.}
  \label{fig:eta-s-geometry}
\end{figure}
}{}

Following on the above explanation, to leverage the Hardy inequality, we reweigh (or ``stretch") the measure $\mu$ by a weight of the form $\prod_j q(|x_j|)$, for a carefully chosen $q:\R_{\geq 0}\to\R_{\geq 0}$, and obtain a new measure $\eta$ (see also Figure \ref{fig:eta-s-geometry}).\footnote{In the proof, $\eta$ is tuned by an appropriate dimension-free parameter $S>0$. Here, we denote $\eta=\eta_S$ for simplicity.} The measure $\eta$ satisfies $4$ important properties:
\begin{enumerate}
    \item If $L$ is the Langevin operator associated to $\eta$, there is a slight modification of the operator $A$ from \eqref{eq:Aheuristic} so that a similar estimate to \eqref{eq:AtoL} holds.
    \item The operator $A$ is a non-negative self-adjoint operator on $L^2(\eta)$, with a discrete spectrum such that
\begin{equation}\label{eq:spectral}
    \inf\limits_{g:\mathrm{Var}_\eta(g)>0} \frac{\E_\eta\left[gAg\right]}{\mathrm{Var}_\eta(g)} > 0.
\end{equation}
   \item It holds that $\E_\eta\left[gAg\right] \simeq \E_\eta\left[gLg\right]$ for every test function $g$.
   \item The measure $\eta$ is close, in an appropriate sense, to $\mu$ so that $C_{\mathrm{P}}(\mu)\leq C_{\mathrm{P}}(\eta)$.
\end{enumerate}
With these $4$ properties, we choose $f$ to be an eigenfunction of $A$, achieving the infimum in \eqref{eq:spectral} and satisfying the symmetries above. We then exploit the additional terms in \eqref{eq:AtoL}, and use integration by parts, to bound $E_{\eta}[fAf]$ from below. The third property then implies a bound on $ C_{\mathrm{P}}(\eta)$, and the fourth property transfers it to  $C_{\mathrm{P}}(\mu)$. It is a curious fact of the proof that $\eta$ is not, and cannot, be log-concave, yet we use it to establish a Poincar\'e inequality for a log-concave measure. The situation is usually reversed in many cases.

We give the precise details of our construction in Section \ref{sec:main-proof}. Here, let us demonstrate one example of a weight $q$ which appears in the literature and satisfies the first $3$ properties we listed above. For this example we choose $q(r) = \sqrt{r}$, so that if $\mu = e^{-V}dx$ we would define $\eta$ to have a density proportional to $\sqrt{\prod_j |x_j|}e^{-V}.$ In this case, since we need to add $-\nabla \log q$ to the definition of $L$, \eqref{eq:Aheuristic} becomes
$$Af = -\Delta f + \langle \nabla V, \nabla f\rangle - \frac{1}{2}\sum_j \frac{1}{x_j}\partial_j f + \frac{1}{2}\sum\limits_{j}\frac{1}{2x_j^2}(f-R_jf).$$
Concretely, note that $Ax_i = \partial_i V$, and so at least when $f$ is linear we can expect $Af \in L^{2}(\mu)$, perhaps after imposing some regularity assumptions on $V$.
The last two terms together, $Tf := - \frac{1}{2}\sum_j \frac{1}{x_j}\partial_j f + \frac{1}{2}\sum\limits_{j}\frac{1}{2x_j^2}(f-R_jf),$ appear in so-called Dunkl theory, and $\Delta - T$ is also called the Dunkl Laplacian. Such operators were first introduced by Dunkl in \cite{Dunkl1989} to study spherical harmonics associated
with finite reflection groups. See also \cite{rosler2003dunkl} for the key role of Dunkl operators in the explicit solvability of quantum integrable systems. The models considered there possess symmetries associated with finite reflection groups, such as particle-exchange symmetry, which are naturally encoded in the Dunkl operators. By now, there is a rich literature on the spectral properties of these operators \cite{rosler2003dunkl}, as well as on related functional inequalities \cite{velicu2022logarithmic}. In our case, the natural symmetry group of an unconditional measure is $\mathbb{Z}_2^n$, for which the orthants serve as Weyl chambers. Importantly, the weight $\sqrt{\prod_j |x_j|}$ vanishes on the boundaries of the Weyl chambers, yet the vanishing is slow enough so as allow for the Hardy inequality to hold which can lead to a possible spectral gap, as in the third property.
The problem for our case is that the weight $\sqrt{\prod_j |x_j|}$ blows up at infinity, which obstructs the fourth property we need. For our proof, we will modify the weight to keep it bounded, which is the missing piece. This modification will require some non-trivial changes to the operator $A$ to ensure it remains self-adjoint.

\subsection{Comparison with the $O(\log^2 (n))$ bound from~\cite{klartag2009berry}}
Klartag in ~\cite{klartag2009berry} also uses a spectral argument which exploits the unconditional symmetries to obtain the $O(\log^2 n)$ bound. Specifically, Bochner-type identities for the Neumann Laplacian show that an eigenfunction corresponding to its first positive eigenvalue can be chosen odd in one coordinate.  Applying a one-dimensional Poincar\'e inequality on coordinate slices then yields
\[
C_\mathrm{P}(\operatorname{Unif}(K))\le CR^2
\]
for every unconditional convex body $K\subseteq[-R,R]^n$ and any $R>0$. For isotropic unconditional convex bodies, one can truncate at $R=O(\log(n))$ to obtain
\[
C_\mathrm{P}(\operatorname{Unif}(K))=O(\log^2(n));
\]
see \cite[Corollaries~2--3 and Equation~(22)]{klartag2009berry}.

As mentioned above, our argument likewise carefully exploits the coordinate parities of an eigenfunction corresponding to the spectral gap. A principal difference is though, the measure and operator being analyzed: we first stretch the measure coordinatewise and then study a Dunkl--Langevin operator adapted to the new measure.

\subsection{Statement on Artificial Intelligence (AI) use}
The core ideas underlying the proofs presented here were generated by AI generative models. The authors initially spent approximately three weeks exploring an information theoretic approach to proving Theorem~\ref{thm:main} in collaboration with the models. Around then, the model independently proposed replacing several key steps with a new argument based on Dunkl theory and specifically the Dunkl--Langevin operator. Over the following four days of intensive but relatively high-level interactions, the model developed a complete proof sketch that differed substantially from the approach explored during the first three weeks. Given the importance of the result and the (to the best of the authors' knowledge) novelty of the method, the authors decided to examine the sketch in detail and digest the proof. As the sketch was difficult to understand, the authors spent several additional weeks working through its details, verifying and refining the argument. This led to a simplified and more accessible proof which we present here. The authors take full responsibility for the content and accuracy of this work.

\subsection{Organization}
In Section~\ref{sec:main-proof}, we deduce Theorem~\ref{thm:main}
from a key auxiliary result Theorem~\ref{thm:bounded} under additional regularity assumptions.
Section~\ref{sec:auxiliary-proof} proves Theorem~\ref{thm:bounded}.
Appendix~\ref{sec:analyticdetails} supplies a formal treatment of the Dunkl operators we use.

\section{Proof of Theorem~\ref{thm:main}}\label{sec:main-proof}

\subsection{The transformation of $\mu$} We start with introducing a transformation of $\mu.$ This transformation plays the role of the reweighting mechanism we mentioned in Section \ref{sec:ideas} and constitutes one of the main components of the proof. 

Fix $S > 0$ and define the following functions
\begin{equation}\label{eq:coordinatechange}
 q_S(r)=\left(\frac r{\sqrt{S^2+r^2}}\right)^{1/2}, r>0, \qquad
 g_S(r)=\int_0^r q_S(t)\dif t,\qquad
 (T_S(y))_j=\sgn(y_j)g_S(|y_j|),
\end{equation}
where $r>0$, $y \in \R^n$, and $j \in[n]$. We think about the map $T_S:\R^n\to \R^n$ as a transport map whose image is $\mu$. We then consider its source which we denote as 
$$\eta_S:=(T_S^{-1})_\#\mu.$$
By a change of variables we see that $\eta_S$ has the density (see also Figure \ref{fig:eta-s-geometry}),
\begin{equation}\label{eq:pullbackdensity}
 \frac{d\eta_S}{dy}(y)\propto e^{-W_S(y)}\prod_{j=1}^nq_S(|y_j|),
 \qquad W_S:=V\circ T_S.
\end{equation}
We remark that since $0<q_S(r)<1$ when $r>0$ and since
$\lim_{r \to \infty} q_S(r)=1$, the map $T_S$ is indeed invertible. Moreover,
\begin{equation}\label{eq:Tlip}
 \norm{DT_S(y)}_{\mathrm{op}}
 =\max_{j\in[n]}g_S'(|y_j|)
 =\max_{j\in[n]}q_S(|y_j|)\le1.
\end{equation}
Thus $T_S$ is also $1$-Lipschitz for every $S>0$.

\subsection{A dichotomy result for $C_{\mathrm{P}}(\eta_S)$} Because the map $T_S$, is  $1$-Lipschitz, to prove Theorem \ref{thm:main} it will suffice to prove that for some universal constant $C>0$, and for some $S>0$, $C_P(\eta_S) \leq C $ (see Lemma \ref{lem:Liptrans} below for a proof of this elementary fact). The following result is our main auxiliary result pertaining to the measure $\eta_S$. It is a dichotomy result which allows us to bound $C_P(\eta_S)$ under certain conditions.

\begin{theorem}\label{thm:bounded}
Fix a scale parameter $S>0$, and let $W\in C^2(\R^n)$ be
unconditional and satisfy
\begin{equation}\label{eq:operator-assumptions}
 0\preceq\hessian W\preceq KI,\qquad
 x\cdot\grad W(x)\ge\kappa_0|x|^2-B_0
 \quad(x\in\R^n),
\end{equation}
for some $0<K<\infty$, $0\le B_0<\infty$, and $\kappa_0>0$.
Define the probability measure $\eta_S$ by
\begin{equation*}
 \frac{d\eta_S}{dx}(x)\propto e^{-W(x)}
\prod_{i=1}^n q_S(|x_i|),
\end{equation*}
with $q_S$ as in \eqref{eq:qchoices}.
%and set $a:=\norm{\Cov(\eta)}_{\mathrm{op}}$.
Then $C_{\mathrm{P}}(\eta_S) < \infty$, and there exists a universal constant $c_*>0$ such that if $S^2 \geq \norm{\Cov(\eta_S)}_{\mathrm{op}}$, then at least one of the following conditions must hold:
\begin{itemize}
    \item Either $ C_\mathrm{P}(\eta_S)>\frac{S^2}{3}$;
    \item or $C_\mathrm{P}(\eta_S)\leq 9\frac{\norm{\Cov(\eta_S)}_{\mathrm{op}}}{c_*}$.
\end{itemize}
%If $S^2\ge a$, then it holds either \[    C_P(\eta)>S^2/3,\] or  \[C_P(\eta)\le9a/c_*. \]
\end{theorem}
Our main use of Theorem \ref{thm:bounded} will be in choosing $S$ in a way which forces $S^2\geq \norm{\Cov(\eta_S)}_{\mathrm{op}}$ while ruling out the possibility that $C_\mathrm{P}(\eta_S)>\frac{S^2}{3}$. In the course of showing the first condition we will also show that $\norm{\Cov(\eta_S)}_{\mathrm{op}}$ cannot be too large.

We defer the proof of Theorem \ref{thm:bounded} to Section~\ref{sec:auxiliary-proof}, and for now focus on using it to prove Theorem \ref{thm:main}.

\begin{remark}
    The use of a dichotomy to obtain a spectral gap bound in Theorem~\ref{thm:bounded} is closely related to multiple bootstrap arguments for spectral gaps in the literature. For example, Hastings~\cite{Hastings2019} studies how a gap between eigenvalues behaves in a perturbation $H_s$ of a self-adjoint matrix $H_0.$ Specifically, a key step towards proving that the spectral gap of $H_s$ remains large for all $s,$ is that he considers a path \(H_s=H_0+sV, s \in [0,1]\) and shows, under suitable perturbative assumptions, that the bound \(\Delta(s)\ge 1/2\) throughout the path improves to \(\Delta(s)\ge 3/4\), where \(\Delta(s)\) denotes the spectral gap of $H_s$. Related bootstrap mechanisms appear in the study of spectral gap bounds for conditioned spin systems of Adhikari--Brennecke--Xu--Yau~\cite{ABXY2024} and in the stability proof of spectral gaps for infinite quantum spin systems of Nachtergaele--Sims--Young~\cite{NSY2024}. 
\end{remark}

%We now show how Theorem \ref{thm:bounded} implies a Poincar\'e inequality for the transformed measure $\eta_S$ and use a simple transportation argument to obtain a similar bound for $\mu$ and complete the proof of Theorem \ref{thm:main}.

\subsection{Auxiliary results}
To apply Theorem \ref{thm:bounded} we need to verify that $\eta_S$ meets the requirements in the statement. First, it is clear that the weight
$$\prod_{j=1}^nq_S(|y_j|),$$
from \eqref{eq:pullbackdensity} is precisely the one appearing in Theorem \ref{thm:bounded}. Thus we will need to show that $W_S$ satisfies the assumptions in \eqref{eq:operator-assumptions}, for which we require a further, but temporary, assumption.

\textbf{An extra assumption:} To apply Theorem \ref{thm:bounded} we first impose an extra regularity condition on $V$. Later in the proof of Theorem \ref{thm:main} we will show how to approximate $V$ by regular enough functions and remove this assumption. Specifically, below we will assume that there exist constants $0<\kappa\leq K<\infty$ such that
\begin{equation}\label{eq:regularoriginal}
 V\in C^\infty,\quad 
 \kappa \mathrm{I}\preceq\hessian V\preceq K\mathrm{I}, \text{ and }
 \max_j\E_\mu[X_j^2]\leq2.
\end{equation}
We emphasize that the bound on $C_\mathrm{P}(\mu)$ we will obtain under \eqref{eq:regularoriginal} is independent of the constants $\kappa$ and $K$, which is the key fact that will allow us to approximate an arbitrary convex and unconditional $V$ with ones satisfying the regularity assumptions and complete the proof of Theorem \ref{thm:main} without them. The assumption on the second moment is for convenience and will allow us to handle, in the approximation step, measures which are not necessarily isotropic.

Under \eqref{eq:regularoriginal}, we verify the hypotheses of Theorem~\ref{thm:bounded} for $\eta_S$ and bound its covariance.
\begin{lemma}\label{lem:pullbackregularity}
Suppose $V$ satisfies \eqref{eq:regularoriginal}. Then, $W_S$ is unconditional and
satisfies \eqref{eq:operator-assumptions}. More precisely,
\begin{equation}\label{eq:pullbackhessianbounds}
 \kappa\diag\bigl(q_S(|y_j|)^2\bigr)
 \preceq\hessian W_S(y)\preceq\frac{3K}{2}I.
\end{equation}
Moreover, for every $S\ge1$,
\begin{equation}\label{eq:pullbackmoments}
 \norm{\Cov(\eta_S)}_{\mathrm{op}}\le14S^{2/3}.
\end{equation}
\end{lemma}
\begin{proof}
First, as $V$ is unconditional and $T_S$ is diagonal it is immediate that $W_S=V\circ T_S$ is unconditional as well. Given this, we will begin by establishing \eqref{eq:pullbackhessianbounds} and then explain how to derive the condition in \eqref{eq:operator-assumptions} from it. 
Differentiating $W_S=V\circ T_S$ gives
$$
\partial_jW_S(y)=\partial_jV(T_S(y))q_S(|y_j|),
$$
and so,
\begin{equation}\label{eq:hessianexp}
    \nabla^2W_S(y)
=DT_S(y)\nabla^2V(T_S(y))DT_S(y)
+\operatorname{diag}\bigl(
\partial_jV(T_S(y))\operatorname{sgn}(y_j)q_S'(|y_j|)
\bigr).
\end{equation}
To obtain a lower bound, fix some $j\in [n]$. Since $V$ is even in the $j^{th}$ coordinate, \eqref{eq:regularoriginal} gives
\begin{equation} \label{eq:Vbound}
    \frac{\partial_jV(x)}{x_j}
=\int_0^1\partial_{jj}V(x_1,\ldots,s\cdot x_j,\ldots,x_n)ds
\in[\kappa,K],
\end{equation}
whenever $x_j\neq 0$.
Going back to the expressions in \eqref{eq:coordinatechange}, we also have, for $r>0$,
$$
g_S(r)\le rq_S(r),
\qquad
0 < rq_S'(r)=\frac{S^2}{2(S^2+r^2)}q_S(r)
\le\frac12q_S(r).
$$
Since, for every $j \in [n]$, $\mathrm{sgn}(T_S(y)_j) = \mathrm{sgn}(y_j)$, applying the previous two displays now implies the bound
\begin{equation} \label{eq:dervbounds}
    0\leq\partial_jV(T_S(y))\operatorname{sgn}(y_j)q_S'(|y_j|)
\leq K\cdot T_S(y)_j\operatorname{sgn}(y_j)q_S'(|y_j|)= K g_S(|y_j|)q_S'(|y_j|)
\leq\frac K2q_S(|y_j|)^2.
\end{equation}
Thus, since $|\partial_jT_S(y)_j| = q_S(|y_j|)$, from \eqref{eq:regularoriginal} we now get,
\begin{equation} \label{eq:lowerbound}
    \kappa\operatorname{diag}\left(q_S(|y_j|)^2\right) \preceq DT_S(y)\nabla^2V(T_S(y))DT_S(y) \preceq \nabla^2W_S(y),
\end{equation}
where the last inequality comes from \eqref{eq:hessianexp}.
For the upper bound, we again use \eqref{eq:regularoriginal} so that,
\begin{equation}
    DT_S(y)\nabla^2V(T_S(y))DT_S(y) \preceq K(DT_S(y))^2 = K\mathrm{diag}\left(q_S(|y_j|)^2\right).
\end{equation}
The combination of the above display with \eqref{eq:dervbounds} and \eqref{eq:lowerbound} implies
$$
\kappa\operatorname{diag}\bigl(q_S(|y_j|)^2\bigr)
\preceq\nabla^2W_S(y)
\preceq\frac{3K}{2}\operatorname{diag}\bigl(q_S(|y_j|)^2\bigr) \preceq \frac{3K}{2}I,
$$
which is \eqref{eq:pullbackhessianbounds}.

Showing that \eqref{eq:operator-assumptions} is a consequence of \eqref{eq:pullbackhessianbounds} is now rudimentary. Since $W_S$ is unconditional,
$\partial_jW_S=0$ when $y_j=0$. The lower bound in
\eqref{eq:pullbackhessianbounds} gives,
\begin{align*}
y_j\partial_jW_S(y) &= y^2_j\int_0^1 \partial_
{jj}W_S(y_1,\dots,s\cdot y_j,\dots,y_n)ds
\geq\kappa y^2_j\int_0^{1}
q_S(s\cdot |y_j|)^2ds\\
& = \kappa y_j^2\int_0^1 \frac{s|y_j|}{\sqrt{S^2+s^2|y_j|^2}}ds
=\kappa|y_j|\bigl(\sqrt{S^2+y_j^2}-S\bigr)\\
&\geq \kappa(y_j^2-|y_j|S) \geq \kappa\left(y_j^2 - \frac{y_j^2+S^2}{2}\right) = \frac{\kappa}{2}(y_j^2 - S^2).
\end{align*}
Summing over $j$ gives \eqref{eq:operator-assumptions}
$$
\langle y,\nabla W_S(y)\rangle
\ge\frac{\kappa}{2}|y|^2-\frac{\kappa nS^2}{2},
$$
with $\kappa_0 = \frac{\kappa}{2}$ and $B_0 = \frac{\kappa nS^2}{2}.$

To finish the proof it remains to show \eqref{eq:pullbackmoments}. For $0\leq r\leq S$,
$$
g_S(r)
\geq \frac{1}{2^{1/4}\sqrt{S}}\int_0^r \sqrt{u}du
\geq \frac{1}{3\sqrt{S}}r^{3/2}.
$$
On the other hand, when $r\geq S$, monotonicity of $q_S$ gives
$$
g_S(r)
\geq g_S(S)+\frac{r-S}{2^{1/4}}
\geq\frac{r-\frac S3}{2^{1/4}}
\ge\frac r2.
$$
Together, we can see that for every $r\geq0$,
$$
r^2\leq4g_S(r)^2+3S^{2/3}g_S(r)^{4/3}.
$$
Since $(T_S)_\#\eta_S=\mu$ and $|T_S(y)_j|=g_S(|y_j|)$, we use the previous inequality and compute for every $j\in[n]$
\begin{align*}
\E_{\eta_S}[Y_j^2]
\leq4\E_{\eta_S}[T_S(Y)_j^2]
+3S^{2/3}\E_{\eta_S}[|T_S(Y)_j|^{4/3}]=4\E_\mu[X_j^2]
+3S^{2/3}\E_\mu[|X_j|^{4/3}].
\end{align*}
By Jensen's inequality and \eqref{eq:regularoriginal},
$
\E_\mu[|X_j|^{4/3}]
\leq\bigl(\E_\mu[X_j^2]\bigr)^{2/3}
\leq2^{2/3}.
$
Therefore, for $S\ge1$,
$$
\E_{\eta_S}[Y_j^2]
\leq 8+3\cdot2^{2/3}S^{2/3}
\leq 14S^{2/3}.
$$
Since $\eta_S$ is unconditional $\E_{\eta_S}[Y_jY_i] = 0$ whenever $i \neq j$, so
$$
\|\operatorname{Cov}(\eta_S)\|_{\mathrm{op}}
=\max\limits_j\E_{\eta_S}[Y_j^2]
\leq 14S^{2/3}.
$$
\end{proof}

\subsection{Useful tools and estimates}
As we explained, our proof proceeds by using Theorem~\ref{thm:bounded} to bound $C_P(\eta_S)$ for some universal constant $S>0$. We will then transfer this bound to $\mu$. 
For the latter step, we use the following well-known fact about the behaviour of the Poincar\'e constant under Lipschitz change of variables.

\begin{lemma}\label{lem:Liptrans}
Let $\nu$ be a probability measure on $\R^n$, and let
$T:\R^n\to\R^n$ be $L$-Lipschitz for some $L>0$.
If $\nu'=T_\#\nu$, then
\[
 C_\mathrm{P}(\nu')\le L^2C_\mathrm{P}(\nu).
\]
\end{lemma}
\begin{proof}
For every test function $f$,
\begin{align*}
 \Var_{\nu'}(f)
 =\Var_\nu(f\circ T)&\leq C_\mathrm{P}(\nu)\E_\nu\left[\norm{\grad(f\circ T)(y)}^2\right]\leq L^2C_\mathrm{P}(\nu)\E_\nu\left[\norm{\grad f(T(y))}^2\right]\\
 &=L^2C_\mathrm{P}(\nu)\E_{\nu'}\left[\norm{\grad f(x)}^2\right].
\end{align*}
\end{proof}
Since Theorem \ref{thm:bounded} can only offer a dichotomy, rather than a universal bound, we also need a way to understand how $C_p(\eta_S)$ varies when $S$ changes. The following lemma gives the required estimate.
\begin{lemma}\label{lem:scale-comparison}
For $0<S'\le S$,
\begin{equation}\label{eq:lipbound}
 C_\mathrm{P}(\eta_{S'})
 \le C_\mathrm{P}(\eta_S)
 \le\frac{S}{S'}C_\mathrm{P}(\eta_{S'}).
\end{equation}
\end{lemma}
\begin{proof}
Set $H_S^{S'}=T_{S'}^{-1}\circ T_S$, so that
$(H_S^{S'})_\#\eta_S=\eta_{S'}$. We show that $H_S^{S'}$ is
$1$-Lipschitz and that its inverse is $\sqrt{\frac{S}{S'}}$-Lipschitz.
Lemma~\ref{lem:Liptrans} will then give \eqref{eq:lipbound}.

Since the maps $T_S$ are symmetric, it suffices to estimate
their derivatives in the positive orthant. Moreover, looking at the expressions in \eqref{eq:coordinatechange}, we can see that $q_S\leq q_{S'}$ and that
$$
 g_{S'}\left(\frac{S'r}{S}\right)=\frac{S'}Sg_S(r).
$$
So, for $r\ge0$,
\[
 \frac{S'}S r\le g_{S'}^{-1}(g_S(r))\le r.
\]
Moreover, $q_{S'}$ is increasing and
$q_{S'}\left(\frac{S'r}{S}\right)=q_S(r)$.
Consequently,
\begin{align*}
 \sqrt{\frac{S'}S}
 \leq\left(\frac{S'^2+r^2}{S^2+r^2}\right)^{1/4}
 =\frac{q_S(r)}{q_{S'}(r)}\le\frac{q_S(r)}{q_{S'}(g_{S'}^{-1}(g_S(r)))}\leq1,
\end{align*}
For $r>0$, the inverse function theorem gives
$$
 (g_{S'}^{-1}\circ g_S)'(r)
 =\frac{q_S(r)}{q_{S'}(g_{S'}^{-1}(g_S(r)))}.
$$
Since $(H_S^{S'}(x))_j$ is a function only of $x_j$, this leads to
$$
 \sqrt{\frac{S'}S}
 \le\left|\partial_jH_S^{S'}(x)\right|
 =\left|(g_{S'}^{-1}\circ g_S)'(x_j)\right|\le1,
$$
from which we can deduce the bounds
$$
 \norm{DH_S^{S'}}_{\mathrm{op}}\leq1 \quad \text{ and } \quad 
 \norm{(DH_S^{S'})^{-1}}_{\mathrm{op}}\leq\sqrt{\frac{S}{S'}}.
$$
Hence $H_S^{S'}$ is $1$-Lipschitz and $(H_S^{S'})^{-1}$ is $\sqrt{\frac{S}{S'}}$-Lipschitz. By Lemma \ref{lem:Liptrans} we are done.
\end{proof}

\subsection{Putting things together}\label{sec:fixed-scale}

\begin{proof}[Proof of Theorem \ref{thm:main}]
We begin under the assumption that $V$ satisfies \eqref{eq:regularoriginal}. Let $c_*$ be the constant from Theorem \ref{thm:bounded} and write $a_S=\norm{\Cov(\eta_S)}_{\mathrm{op}}$.
According to the bound \eqref{eq:pullbackmoments} from Lemma \ref{lem:pullbackregularity} we have
 $a_S \leq 14S^{2/3}$. Note that if $S \geq 14^{4/3}$ then,
\begin{equation}\label{eq:abound}
    a_S\leq14S^{2/3}\leq S^2.
\end{equation}
which is required for Theorem \ref{thm:bounded}. Thus for $S \geq 14^{4/3}$, Lemma \ref{lem:pullbackregularity} shows we can apply Theorem \ref{thm:bounded} to $\eta_S$ from which we get the implication, 
\begin{equation} \label{eq:dicho}
 C_{\mathrm{P}}(\eta_S)\leq\frac{S^2}{3}
 \quad\Longrightarrow\quad
 C_{\mathrm{P}}(\eta_S)\leq\frac{9a_S}{c_*}.
\end{equation}
We now set $S_* = \max\left(14,3\frac{126}{c_*}\right)^{4/3}$ and claim that 
\begin{equation} \label{eq:Sstarbound}
    C_\mathrm{P}(\eta_{S_*}) \leq \frac{S^2_*}{3}.
\end{equation}
%In light of \eqref{eq:dicho} and \eqref{eq:abound} we now set $S_* = \max\left(14,\frac{126}{c_*}\right)^{4/3}$ for which we can see that, if $C_\mathrm{P}(\eta_{S_*}) \leq \frac{S_*}{3}$, then
%$$C_\mathrm{P}(\eta_{S_*}) \leq \frac{9a_{S_*}}{c_*}\leq \frac{126S_*^{2/3}}{c_*}\leq S_*^2.$$
%So, it remains to rule out the possibility that $C_\mathrm{P}(\eta_{S_*}) > \frac{S_*}{3}$.
Indeed, suppose towards a contradiction, that this is not the case and choose
$$
 S'=\frac{3C_\mathrm{P}(\eta_{S_*})}{S_*}>S_*.
$$
In this case Lemma \ref{lem:scale-comparison} yields
$$
 C_\mathrm{P}(\eta_{S'}) \leq\frac{S'}{S_*}C_\mathrm{P}(\eta_{S_*}) =\frac{S'^2}{3}.
$$
Applying the implication in \eqref{eq:dicho} to $S'$, we obtain
$$
 C_\mathrm{P}(\eta_{S_*})\le C_\mathrm{P}(\eta_{S'})\leq 9\frac{a_{S'}}{c_*}\leq \frac{126}{c_*}S'^{2/3}=
 \frac{126}{c_*}\left(\frac{3C_\mathrm{P}(\eta_{S_*})}{S_*}\right)^{2/3}.
$$
Rearranging terms gives
$$
 C_\mathrm{P}(\eta_{S_*})\leq\frac{9\cdot126^3}{S_*^2c_*^3}\le\frac{S_*^2}{3},
$$
where the last inequality follows since $S_*^{4/3}\geq3\frac{126}{c_*}$.
This contradicts our assumption on  $C_\mathrm{P}(\eta_{S_*})$, and so we have established \eqref{eq:Sstarbound}. Finally, $(T_{S_*})_\#\eta_{S_*}=\mu$ and $T_{S_*}$ is
$1$-Lipschitz, as in \eqref{eq:Tlip}. Lemma~\ref{lem:Liptrans} therefore gives
\begin{equation}\label{eq:regularPI}
 C_P(\mu)\le C_\mathrm{P}(\eta_{S_*})\leq\frac{S_*^2}{3}
\end{equation}
which is a universal constant.

To finish the proof, we now remove assumption \eqref{eq:regularoriginal}. Let $\mu=e^{-V}dx$ be an arbitrary
isotropic unconditional log-concave probability measure.
For $\eps>0$, let $\gamma_\eps$ be the Gaussian density with variance $\eps$. Consider the measure $\mu_\eps$ with density proportional to $e^{-\eps|x|^2}\left(e^{-V}*\gamma_\eps\right)$, and write $V_\eps = -\log\left(\frac{d\mu_\eps}{dx}\right)$. Clearly $V_\eps$ is unconditional and $V_\eps \in C^\infty$. Moreover, it is straightforward to check 
$$
2\eps I\preceq\nabla^2V_\eps\preceq \left(\frac{1}{\eps}+2\eps\right)I.
$$
The regularization $\mu_\eps$ was also considered in \cite[Section~4]{klartag2025isoperimetric} where they observed that 
$$
\operatorname{Cov}(\mu_\eps)\longrightarrow \operatorname{Cov}(\mu)=\mathrm{I}\quad
\text{ and } \quad
C_\mathrm{P}(\mu)\leq\liminf_{\eps\to0}C_\mathrm{P}(\mu_\eps).
$$
Since each $\mu_\eps$ is unconditional and hence centered, the above shows that for $\eps$ sufficiently small covariance convergence implies
that $\max_i\E_{\mu_\eps}[X_i^2] = \|\mathrm{Cov}(\mu_\eps)\|_{\mathrm{op}}\leq2$.
Thus for $\eps$ sufficently small, $\mu_\eps$ satisfies \eqref{eq:regularoriginal} and $C_\mathrm{P}(\mu_\eps)\leq \frac{S_*^2}{3}.$ We can now conclude from the above,
$$
C_\mathrm{P}(\mu)\leq\liminf_{\eps\to0}C_\mathrm{P}(\mu_\eps)
\leq\frac{S_*^2}{3}.
$$
\end{proof}

\section{A spectral gap for the auxiliary measure: a proof of the key dichotomy result}\label{sec:auxiliary-proof}

Let $W\in C^2(\R^n)$ be unconditional and satisfy
\eqref{eq:operator-assumptions}. Below we will consider the following functions: for $r\ge0$
\begin{equation}\label{eq:qchoices}
\begin{aligned}
 q_S(r):=q(r)&=\left(\frac r{\sqrt{S^2+r^2}}\right)^{1/2},
 &c_S(r) = c(r)&=r \frac{q'(r)}{q(r)}=\frac{S^2}{2(S^2+r^2)},\qquad S>0.
\end{aligned}
\end{equation}
The function $q$ is the weight we apply to our measure $\mu$, and we will use $\we$ to adapt the Langevin operator.
With these functions, we make the following definitions
\begin{equation}\label{eq:auxiliary-density}
 d\eta(x)\propto e^{-W(x)}\prod_{j=1}^n q(|x_j|)\,dx,\qquad
 c_j(x)=c(|x_j|), j \in [n],\qquad
 a=\|\Cov(\eta)\|_{\mathrm{op}}=\max_{j \in [n]}\E_{\eta}[X_j^2].
\end{equation}
The last equality holds because of the unconditionality of $W$ which makes $\eta$ centered with a diagonal covariance matrix. Below, we will always use $X$ for a random vector with law $\eta$, and for $j \in[n]$ we will write $W_j := \partial_j W$, as well as $c_j' = \partial_j c_j$.

\subsection{The Dunkl Langevin operator}\label{sec:operator}
The proof of Theorem \ref{thm:bounded} crucially relies on modifying the standard Langevin operator, which gives rise to the standard gradient Dirichlet form $\E_{\eta}\left[|\nabla f|^2\right]$. Our modification adds a so-called Dunkl term, aimed at helping to exploit the symmetries in the measure $\eta$.

To define the modification term, we begin by introducing for every $j \in [n]$ the operator $R_j$, which is the linear reflection operator with respect to coordinate $j$. Using these operators, we define the family of difference operators $D_j$ acting on test functions by
$$D_jf(x):=\frac{f(x)-R_jf(x)}{2x_j} = \frac{f(x_1,\dots,x_j,\dots,x_n)-f(x_1,\dots,-x_j,\dots,x_n)}{2x_j}.$$
The difference operators and the functions $\we_j$ from \eqref{eq:qchoices} now give rise to our modified operator $A$, which we call here the \emph{Dunkl Langevin operator}:
\begin{equation}\label{eq:DunkelRegu}
  Af=-\Delta f+\langle\nabla W,\nabla f\rangle
       -\sum_j\frac{c_j}{x_j}\partial_jf
       +\sum_j\frac{c_j}{2x_j^2}(f-R_jf).
\end{equation}
To provide some intuition for the definition, remark that $-\Delta f+\nabla W\cdot\nabla f$ is the standard Langevin operator associated to $e^{-W}dx$. By adding the term $ -\sum_j\frac{c_j}{x_j}\partial_jf$ it transforms into the Langevin operator of $\eta$. Hence the new term is the one containing $c_j\frac{f-R_jf}{2x_j^2} = \frac{c_j}{x_j}D_jf$. Additionally, we also have,
$$\sum_j\frac{c_j}{2x_j^2}(f-R_jf)-\sum_j\frac{c_j}{x_j}\partial_jf = \sum\limits_j \frac{c_j}{x_j}\left(D_jf - \partial_j f\right).$$
In particular, we have $\frac{D_jf - \partial_j f}{x_j}\xrightarrow{x_j\to 0} -\partial_{jj}f.$ Thus, the additional summation terms can be thought of as a weighted Laplacian term.

Next, we define an associated modified Dirichlet form which arises from the operator $A$,
$$ \cE(f,g)=\E_\eta\left[\langle\nabla f,\nabla g\rangle\right] +\sum_j\E_\eta[\we_j\cdot D_jfD_jg].$$
Note that the above formulas require some regularity and integrability assumptions. In Appendix \ref{sec:analyticdetails} we give a rigorous treatment to the domains of $A$ and $\cE$, and prove that all formulas are well-defined. As a slight abuse of notation, in the sequel we will colloquially refer to a \emph{test function} as a function for which all these formulas hold.

Our first claim is that $\mathcal{E}$ is indeed the Dirichlet form associated to $A$. 
\begin{lemma} \label{lem:diriform}
    Let the above notation prevail, and let$f$ and $g$ be two test functions. Then, 
    $$\E_\eta\left[g\cdot Af\right]=\cE(g,f),$$
    in particular $A$ is symmetric in $L^2(\eta),$ i.e.,
    $$\E_\eta\left[g\cdot Af\right] = \E_\eta\left[Ag\cdot f\right].$$
\end{lemma}
\begin{proof}
For every $j \in[n]$, the density $\eta$ satisfies
$$
\partial_j\eta(x)
=\left(\frac{c_j}{x_j}-W_j\right)\eta(x).
$$
Thus integration by parts gives
\begin{equation} \label{eq:lanegibt}
    \E_\eta\left[
\left(-\partial_{jj}f+
\left(W_j-\frac{c_j}{X_j}\right)\partial_jf\right)g
\right]
=\mathbb E_\eta[\partial_jf\,\partial_jg],
\end{equation}
which is the standard identity for the Langevin operator.
For the reflection term, changing variables $x\to R_jx$ gives
$$
\E_\eta\left[
\frac{c_j}{X_j}D_jfg
\right]=\E_\eta\left[
\frac{c_j}{2X_j^2}(f-R_jf)g
\right]
=
-\mathbb E_\eta\left[
\frac{c_j}{2X_j^2}(f-R_jf)R_jg
\right].
$$
Averaging the two expressions, we obtain
\begin{equation}\label{eq:dunkelibp}
    \E_\eta\left[
\frac{c_j}{2X_j^2}(f-R_jf)g
\right]
=\E_\eta\left[
\frac{c_j}{4X_j^2}(f-R_jf)(g-R_jg)
\right]=\E_\eta[c_jD_jfD_jg].
\end{equation}
Summing over \eqref{eq:lanegibt} and \eqref{eq:dunkelibp} over $j$ proves
$$
\E_\eta\left[g\cdot Af\right]
=\mathbb E_\eta[\langle\nabla f,\nabla g\rangle]
+\sum_j\mathbb E_\eta[c_jD_jfD_jg]
=\mathcal E(f,g).
$$
\end{proof}
Given Lemma \ref{lem:diriform}, we can now see that $A$ is a non-negative operator in the sense that $\E_\eta[f\cdot Af] \geq 0$, and that constant functions lie in $\mathrm{ker}(A)$. It then makes sense to understand the first non-trivial eigenvalue of $A$. If we abbreviate $\cE(f,f) = \cE(f)$, then this eigenvalue is
\begin{equation}\label{eq:reflection-form}
\begin{aligned}
\lambda&=\inf_{\substack{
                          \Var_\eta(f)>0}}
                       \frac{\cE(f)}{\Var_\eta(f)},\\
 \text{where: }\cE(f)&=\E_\eta\left[|\nabla f|^2\right]+
                 \sum\limits_{j=1}^n\E_\eta\left[c_j(D_jf)^2\right].
\end{aligned}
\end{equation}
A key property of the spectral gap $\lambda$ is that it gives good control on $C_\mathrm{P}(\eta)$, the Poincar\'e constant of $\eta$. We show this by establishing that the Dirichlet form $\cE$ is comparable to the standard gradient Dirichlet form $\E_\eta\left[|\nabla f|^2\right]$. As we explained in Section \ref{sec:ideas} this is essentially a weighted Hardy-type inequality, similar to\cite{velicu2021hardy}. For our specific choice of $q$ we provide the proof below.
\begin{lemma} \label{lem:diriforms}
    Let $f$ be a test function. Then,
    $$\cE(f)\leq 9\cdot \E_\eta\left[|\nabla f|^2\right].$$
    Consequently
    $$\frac{1}{\lambda}\leq C_\mathrm{P}(\eta)\leq\frac{9}{\lambda}.$$
\end{lemma}
\begin{proof}
Fix $j$ and set $h=\frac{f-R_jf}{2}$, so that $\frac{h}{x_j}=D_jf$.
Again using
$$
\partial_j\eta(x) = 
\left(\frac{c_j}{x_j}-W_j\right)\eta(x),
$$
integration by parts gives
$$
2\E_\eta[D_jf\partial_jh] = \E_{\eta}\left[\frac{\partial_j(h^2)}{X_j}\right] = \E_{\eta}\left[\left(\frac{1-c_j}{X^2_j} +\frac{W_j}{X_j}\right)h^2\right] = \E_\eta\left[(1-c_j+X_jW_j)(D_jf)^2\right].
$$
Since $W$ is convex and unconditional, $X_jW_j\geq0$, coupled with $c_j\leq \frac{1}{2}$, and the Cauchy-Schwarz inequality we get
$$
\frac{1}{2}\E_\eta[(D_jf)^2]
\leq2\E_\eta[D_jf\partial_jh]
\leq
2\sqrt{\E_\eta[(D_jf)^2]\,
       \E_\eta[(\partial_jh)^2]}.
$$
Rearranging terms,
$$
\E_\eta[(D_jf)^2]
\leq16\E_\eta[(\partial_jh)^2]
\leq16\E_\eta[(\partial_jf)^2],
$$
where the last inequality follows since $\eta$ is unconditional and since
$$
\partial_jh=\frac{\partial_jf+R_j\partial_jf}{2}.
$$
Summing over $j$ and using $c_j\leq\frac{1}{2}$ again, we obtain
\begin{align*}
\cE(f)
=\E_\eta[|\nabla f|^2]
+\sum_j\E_\eta[c_j(D_jf)^2]\leq\E_\eta[|\nabla f|^2]
+8\sum_j\E_\eta[(\partial_jf)^2]
=9\E_\eta[|\nabla f|^2].
\end{align*}
Together with the trivial inequality $\cE(f)\geq\E_\eta[|\nabla f|^2]$,
from \eqref{eq:reflection-form} we arrive at,
$$
\frac{1}{\lambda}\leq C_\mathrm{P}(\eta)\leq\frac{9}{\lambda}.
$$
\end{proof}
\subsection{Symmetries of the eigenfunction and the proof of Theorem~\ref{thm:bounded}}\label{sec:dichotomy}
The next proposition characterizes a series of possible symmetries of an eigenfunction corresponding to $\lambda.$ Using these symmetries, we establish a dichotomy for the potential values of $\lambda.$
\begin{proposition}\label{prop:keyestimate}
Let $W\in C^2(\R^n)$ be unconditional and satisfy
\eqref{eq:operator-assumptions} and let $\eta$ and $\lambda$ be as in
\eqref{eq:auxiliary-density} and \eqref{eq:reflection-form}.
Then the following hold:
\begin{enumerate}
    \item There exists an $i \in[n]$ and an eigenfunction $f$, such that
    \begin{equation*}
     Af=\lambda f,\qquad \E_\eta[f^2]=1,\qquad
 R_if=-f,\qquad R_jf=f\quad(j\ne i),\qquad x_if\ge0
 \quad\eta\text{-almost everywhere}.
    \end{equation*}
    \item If $\lambda\geq \frac{3}{S^2}$, then 
\begin{equation}\label{eq:Jupper}
 \lambda \ge \left(\frac{\E_\eta\left[c_i\right]}{2\E_\eta\left[X_if\right]^2}\E_\eta\left[c_i\frac{W_if^2}{X^3_i}\right]\right)^{1/3}.
\end{equation}
\end{enumerate}
\end{proposition}
The first item of Proposition \ref{prop:keyestimate} should be compared to the guarantee in \cite[Corollary 13]{barthe2020spectral}. In that work, similar results were obtained for log-concave measures. Here, however, $\eta$ is not log-concave. Instead, we will obtain those symmetries through the choice of the functions $\we_j$. We use the fact that $\we_j>0$ to exclude eigenfunctions that are odd in two or more coordinates, and $\frac{\we_j'}{x_j}<0$ to exclude eigenfunctions that are even in every coordinate.
Using this extra information about symmetry then leads to the bound in the second item.

With Proposition~\ref{prop:keyestimate} at hand, we can now deduce
Theorem~\ref{thm:bounded}.

\begin{proof}[Proof of Theorem~\ref{thm:bounded}]
Set $c_*=\frac{1}{190}$ and assume that $C_{\mathrm P}(\eta)\leq \frac{S^2}{3}$,
otherwise there is nothing to prove.
By Lemma \ref{lem:diriforms}, $\lambda\geq \frac{1}{C_{\mathrm{P}}(\eta)} \geq \frac{3}{S^2}$, so both items of Proposition \ref{prop:keyestimate} apply, and we choose $f$ to be the eigenfunction guaranteed by the proposition.

A direct calculation shows that  $Ax_i=W_i$, and so since $A$ is self-adjoint
$$
\lambda\E_\eta[X_if]= \E_\eta\left[X_iAf\right] = \E_\eta[W_if].
$$
Apply the Cauchy-Schwarz inequality to the right hand side and use the second item of Proposition \ref{prop:keyestimate},
\begin{align*}
\lambda^2\E_\eta[X_if]^2
\le
\E_\eta\left[c_i\frac{W_if^2}{X_i^3}\right]
\E_\eta\left[\frac{W_iX_i^3}{c_i}\right]
\le
\frac{2\lambda^3\E_\eta[X_if]^2}{\E_\eta[c_i]}
\E_\eta\left[\frac{W_iX_i^3}{c_i}\right].
\end{align*}
Rearranging terms gives
\begin{equation} \label{eq:numdeno}
    \lambda\geq
\frac{\E_\eta[c_i]}
{2\E_\eta\left[\frac{W_iX_i^3}{c_i}\right]}.
\end{equation}
To bound the expectations involving $c_i$, we return to the expression  \eqref{eq:qchoices} from which we can see,
\[
 {c_i}=\left(2\left(1+\frac{X_i^2}{S^2}\right)\right)^{-1},
\]
and the right hand side is a convex function of $X^2_i$.
Hence, for the numerator of \eqref{eq:numdeno}, by Jensen's inequality,
$$
\E_\eta[c_i]
\geq\left(2\left(1+\frac{\E_\eta[X_i^2]}{S^2}\right)\right)^{-1}
\geq\frac38,
$$
since, by assumption, $\E_\eta[X_i^2]\leq C_{\mathrm{P}}(\eta) \leq \frac{S^2}{3}.$
As for the denominator of \eqref{eq:numdeno}, using the same identity, as above, for $c_i,$ gives
\begin{align*}
\E_\eta\left[\frac{W_iX_i^3}{c_i}\right] = 2\E_{\eta}\left[W_iX_i^3\right] + \frac{2}{S^2}\E_\eta[W_iX_i^5].
\end{align*}
We now recall that the definition of $c_i$ from \eqref{eq:qchoices} and use it to integrate by parts, so that
\begin{align*}
2\E_{\eta}\left[W_iX_i^3\right] + \frac{2}{S^2}\E_\eta[W_iX_i^5]=2\E_\eta[(3+c_i)X_i^2]
+\frac2{S^2}\E_\eta[(5+c_i)X_i^4]\leq 10\E_\eta[X_i^2]+\frac{14}{S^2}\E_\eta[X_i^4],
\end{align*}
where we used $c_i\leq 1$. To finish we apply the Poincar\'e inequality \eqref{eq:poincare} to the function $x^2_i$, so that because $\E_\eta[X_i^2]\leq C_{\mathrm{P}}(\eta)$, and because we assumed $C_{\mathrm{P}}(\eta) \leq \frac{S^2}{3}$,
$$
\E_\eta[X_i^4]
\le \E_\eta[X_i^2]^2
+4C_{\mathrm P}(\eta)\E_\eta[X_i^2]
\leq \frac{5}{3}S^2\E_\eta[X_i^2].
$$
Combining with the above now gives
\begin{align*}
\E_\eta\left[\frac{W_iX_i^3}{c_i}\right]
\leq10\E_\eta[X_i^2]+\frac{14}{S^2}\E_\eta[X_i^4] \leq 35\E_\eta[X_i^2].
\end{align*}
Plugging this back into \eqref{eq:numdeno} and using the lower bound for $\E_\eta[c_i]$ yields

$$
\lambda\geq\frac{1}{190\E_\eta[X_i^2]},
\implies
C_{\mathrm P}(\eta)\le\frac9\lambda
\leq\frac{9\|\mathrm{Cov}(\eta)\|_{\mathrm{op}}}{c_*},
$$
where the final inequality is Lemma \ref{lem:diriforms}.
\end{proof}

\subsection{Preliminaries for the proof of Proposition \ref{prop:keyestimate}}\label{sec:identities}
To finish the proof of Theorem \ref{thm:bounded}, we need to prove Proposition \ref{prop:keyestimate}. Here we collect some useful facts about the Dunkl Langevin operator. It is important to note that the results below are facilitated by our use of the weight $q$ in \eqref{eq:qchoices}. In particular, we shall heavily use the fact that the function $\we$ is \emph{strictly positive}, which demonstrates the need for the introduction of the auxiliary weight.

\subsubsection{An a priori spectral gap.} We first establish an a priori bound for unconditional test functions, which follows from simple integrations by parts. The role of the weight $q$ is to enable this initial bound. After establishing appropriate symmetries of the eigenfunction corresponding to $\lambda$ \eqref{eq:reflection-form}, we will be able to bootstrap this bound and obtain the final result.
\begin{lemma}\label{lem:weighted}
Let $W\in C^2(\R^n)$ be unconditional and satisfy
\eqref{eq:operator-assumptions} and let $\eta$ and $\lambda$ as in
\eqref{eq:auxiliary-density} and \eqref{eq:reflection-form}.
Let $g$ be an unconditional test function. Then, for any $i \in [n],$
$$\left(\lambda - \frac{3}{2S^2}\right)\left(\E_\eta[c_ig^2] - \frac{\E_\eta[c_ig]^2}{\E_\eta[c_i]}\right) \leq \E_{\eta}\left[c_i|\nabla g|^2\right].$$
\end{lemma}
Before proving the claim, we remark that one could define the measures $ d\nu_i=\frac{c_i}{\E_\eta[c_i]}d\eta$, in which case the conclusion of Lemma \ref{lem:weighted} is equivalent to 
$$\left(\lambda - \frac{3}{2S^2}\right)\mathrm{Var}_{\nu_i}(g) \leq \E_{\nu_i}\left[|\nabla g|^2\right],$$
for unconditional functions.
\begin{proof}
If $\lambda\leq \frac{3}{2S^2}$, then there is nothing to prove. 
Otherwise, fix $i$ and for an unconditional test function $g$, set
$$
h=\sqrt{c_i}\left(
g-\frac{\E_\eta[\sqrt{c_i}g]}
{\E_\eta[\sqrt{c_i}]}
\right).
$$
Then, since both $c_i$ and $g$ are unconditional, the same is true for $h$ and furthermore $\E_\eta[h]=0$,
so that for our claim it shall suffice to prove,
$$\left(\lambda-\frac{3}{2S^2}\right)\E_\eta[h^2] \leq \E_\eta[c_i|\nabla g|^2].$$
Indeed, this follows since, by opening up the definition of $h$, we can see,
$$\E_\eta[c_ig^2]-\frac{\E_\eta[c_ig]^2}{\E_{\eta}[c_i]}=\E_\eta[h^2] - \frac{\E_\eta[\sqrt{c_i}h]^2}{\E_{\eta}[c_i]}\leq \E_\eta[h^2].$$
As $h$ is unconditional, we have $D_jh = 0$ for every $j \in [n]$, and using  \eqref{eq:reflection-form} and the definition of $\lambda$, 
\begin{equation} \label{eq:dirichletbound}
\lambda\E_\eta[h^2]
\leq\cE(h)
=\E_\eta[|\nabla h|^2].
\end{equation}
Starting from the expressions in \eqref{eq:qchoices} we compute,
$$
\partial_i\sqrt{\we_i}
=-\frac{x_i}{S^2+x_i^2}\sqrt{\we_i},
$$
and furthermore, since $\we_i$ depends only on coordinate $i$,
$$
\partial_jh =\sqrt{\we_i}\partial_jg
\quad\text{ for } j\neq i, \quad \text{ and }\quad
\partial_ih=\sqrt{c_i}\partial_ig - \frac{x_i}{S^2+x_i^2}h.
$$
Consequently,
\begin{align*}
c_i|\nabla g|^2
=\sum_{j\neq i}(\partial_jh)^2
+\left(\partial_ih+\frac{x_i}{S^2+x_i^2}h\right)^2\
=|\nabla h|^2
+\frac{x_i}{S^2+x_i^2}\partial_i(h^2)
+\frac{x_i^2}{(S^2+x_i^2)^2}h^2.
\end{align*}
Note that
$$
\partial_i\left(\frac{x_i}{S^2+x_i^2}\right)
=\frac{S^2-x_i^2}{(S^2+x_i^2)^2}
\quad \text{ and }\quad
\partial_i\eta(x) = \left(\frac{c_i}{x_i}-W_i\right)\eta(x).
$$
So, integration by parts gives
\begin{align*}
\E_\eta\left[
\frac{X_i}{S^2+X_i^2}\partial_i(h^2)
\right]
&=-\E_\eta\left[
\left(
\frac{S^2-X_i^2}{(S^2+X_i^2)^2}
+\frac{c_i}{S^2+X_i^2}
-\frac{X_iW_i}{S^2+X_i^2}
\right)h^2
\right]\\
&=-\E_\eta\left[
\left(
\frac{S^2-X_i^2}{(S^2+X_i^2)^2}
+\frac{S^2}{2(S^2+X_i^2)^2}
-\frac{X_iW_i}{S^2+X_i^2}
\right)h^2
\right],
\end{align*}
where we also used $c_i=\frac{S^2}{2(S^2+x_i^2)}$ from \eqref{eq:qchoices}. 
Combining the above we now get,
\begin{align*}
\E_\eta[c_i|\nabla g|^2]
&=\E_\eta|\nabla h|^2
+\E_\eta\left[
\frac{X_i}{S^2+X_i^2}\partial_i(h^2)
+\frac{X_i^2}{(S^2+X_i^2)^2}h^2
\right]\\
&=\E_\eta|\nabla h|^2
+\E_\eta\left[
\left(
\frac{4X_i^2-3S^2}{2(S^2+X_i^2)^2}
+\frac{X_iW_i}{S^2+X_i^2}
\right)h^2
\right].
\end{align*}
Finally, since $W$ is convex unconditional, $W_iX_i \geq 0$, and therefore
$$\frac{4X_i^2-3S^2}{2(S^2+X_i^2)^2}
+\frac{X_iW_i}{S^2+X_i^2} \geq -\frac{3S^2}{2(S^2+X_i^2)^2} \geq -\frac{3}{2S^2}.$$
We now conclude by combining the above displays with \eqref{eq:dirichletbound},
\begin{align*}
\E_\eta[c_i|\nabla g|^2]
=\E_\eta|\nabla h|^2
+\E_\eta\left[
\left(
\frac{4X_i^2-3S^2}{2(S^2+X_i^2)^2}
+\frac{X_iW_i}{S^2+X_i^2}
\right)h^2
\right]\geq
\left(\lambda-\frac{3}{2S^2}\right)\E_\eta[h^2].
\end{align*}
\end{proof}

\subsubsection{Bochner-type identities}  To prove Proposition \ref{prop:keyestimate} we require some identities connecting potential symmetries of a function $f$ and the Dirichlet form of its derivatives $\cE(\partial_j f)$. These identities are similar to Bochner's formula for the Langevin operator, and we will apply them in a similar fashion to \cite{barthe2020spectral,klartag2009berry}. The computation is essentially a generalization of \cite{li2023dimension}, who considered the case when the functions $c_j$ are constants. 
\begin{lemma}\label{lem:identities}
If $f$ is even in every coordinate, then
\begin{equation}\label{eq:evenidentity}
 \E_\eta[(Af)^2]=\sum_j\cE(\partial_jf)
 +\E_\eta\left[\langle\nabla^2W\nabla f,\nabla f\rangle\right]
 +\sum_j\E_\eta\left[-\frac{c_j'}{x_j}(\partial_jf)^2\right].
\end{equation}
If $f$ is odd in $x_i$ and even in the other coordinates,
put $g=f/x_i$. %with its smooth extension at $x_i=0$. 
Then
\begin{align}\label{eq:oddidentity}
 \E_\eta[(Af)^2]={}&\sum_j\cE(\partial_jf)
 +\E_{\eta}\langle\nabla^2W\nabla f,\nabla f\rangle
 +\sum_{j\ne i}\E_\eta\left[-\frac{c_j'}{x_j}(\partial_jf)^2\right]\notag\\
 &+\E_\eta[c_i|\nabla g|^2]
 +\E_\eta[(2c_i-x_ic_i')(\partial_ig)^2]
 +\E_\eta\left[c_i\frac{W_i}{x_i}g^2\right].
\end{align}
%The identities also hold for $f\in\Dom(A)$ with the stated symmetries.
%Each $\partial_jf$ belongs to $\Dom(\cE)$. In the second identity, $v$
%is a limit of smooth even functions in the norm
%$(\E[c_iv^2]+\E[c_i|\nabla v|^2])^{1/2}$. Every displayed term is finite.
\end{lemma}
\begin{proof}
Suppose first that $f$ is even in the $j^{th}$ coordinate. Then, $R_jf = f,$ and furthermore $\partial_jf$ is odd in the $j^{th}$ coordinate, and so
$$
R_j(\partial_jf)=-\partial_jf
\quad \text{ and }\quad 
\partial_j(R_kf)=R_k(\partial_jf)\text{ for } k\neq j.
$$
Substituting $\partial_jf$ into the definition of $A$ in \eqref{eq:DunkelRegu} gives
\begin{align*}
A(\partial_jf)
={}&-\Delta(\partial_jf)
+\langle\nabla W,\nabla\partial_jf\rangle
-\sum_k\frac{c_k}{x_k}\partial_{kj}f\\
&+\frac{c_j}{x_j^2}\partial_jf
+\sum_{k\ne j}\frac{c_k}{2x_k^2}
\bigl(\partial_jf-R_k(\partial_jf)\bigr).
\end{align*}
On the other hand, differentiating $Af$ gives
\begin{align*}
\partial_j(Af)
={}&-\Delta(\partial_jf)
+\langle\nabla W,\nabla\partial_jf\rangle
+\langle\nabla W_j,\nabla f\rangle
-\sum_k\frac{c_k}{x_k}\partial_{kj}f\\
&+\left(\frac{c_j}{x_j^2}-\frac{c_j'}{x_j}\right)\partial_jf
+\sum_{k\ne j}\frac{c_k}{2x_k^2}
\bigl(\partial_jf-R_k(\partial_jf)\bigr).
\end{align*}
Subtracting these expressions, we obtain
\begin{equation} \label{eq:commrelation}
\partial_j(Af)
=A(\partial_jf)
+\langle\nabla W_j,\nabla f\rangle
-\frac{c_j'}{x_j}\partial_jf.
\end{equation}

Now suppose that $f$ is even in every coordinate.
Then $D_jf=0$ for every $j$, so Lemma \ref{lem:diriform} and the preceding
identity give
\begin{align*}
 \E_\eta[(Af)^2]
&=\cE(f,Af)
=\sum_j\E_\eta[\partial_jf\,\partial_j(Af)]\\
&=\sum_j\E_\eta[\partial_jf\,A(\partial_jf)]
+\E_\eta[\langle\nabla^2W\nabla f,\nabla f\rangle]
+\sum_j\E_\eta\left[-\frac{c_j'}{X_j}(\partial_jf)^2\right]\\
&=\sum_j\cE(\partial_jf)
+\E_\eta[\langle\nabla^2W\nabla f,\nabla f\rangle]
+\sum_j\E_\eta\left[-\frac{c_j'}{X_j}(\partial_jf)^2\right],
\end{align*}
where we used $\sum\limits_j \E_{\eta}\left[\partial_jf\langle \nabla W_j,\nabla f\rangle\right] = \sum\limits_{j,k}\E_{\eta}\left[\partial_jf\partial_{kj} W\partial_k f\right] =\E_\eta[\langle\nabla^2W\nabla f,\nabla f\rangle].$
This proves the first identity.

Suppose now that $f$ is odd in the $i^{th}$ coordinate
and even in every other coordinate. Put $g=f/x_i$.
Then both $g$ and $\partial_if$ are unconditional.
Substituting $\partial_if$ into the definition of $A$ gives
$$
A(\partial_if)
=-\Delta(\partial_if)
+\langle\nabla W,\nabla\partial_if\rangle
-\sum_k\frac{c_k}{x_k}\partial_{ki}f.
$$
Since $R_if=-f$ and $R_kf=f$ for $k\ne i$,
differentiating $Af$ gives
\begin{align*}
\partial_i(Af)
={}&-\Delta(\partial_if)
+\langle\nabla W,\nabla\partial_if\rangle
+\langle\nabla W_i,\nabla f\rangle
-\sum_k\frac{c_k}{x_k}\partial_{ki}f\\
&+\left(\frac{2c_i}{x_i^2}-\frac{c_i'}{x_i}\right)\partial_if
+\left(\frac{c_i'}{x_i^2}-\frac{2c_i}{x_i^3}\right)f.
\end{align*}
Subtracting these expressions and using
$f=x_ig$ and $\partial_if=g+x_i\partial_ig$, we obtain the analog of \eqref{eq:commrelation}
$$
\partial_i(Af)
=A(\partial_if)
+\langle\nabla W_i,\nabla f\rangle
+\left(\frac{2c_i}{x_i}-c_i'\right)\partial_ig.
$$
For $j\ne i$, \eqref{eq:commrelation} still applies because
$f$ is even in coordinate $j$.

Observe that since $W$ is unconditional and the functions $c_k$ are symmetric, we have the commutation relation $R_iA=AR_i$. Thus, $R_if=-f$ implies
$$
R_i(Af)=A(R_if)=-Af.
$$
By the definition of $D_i$,
$$
D_i(Af)=\frac{Af-R_i(Af)}{2x_i}=\frac{Af}{x_i}.
$$
Furthermore, $D_if= \frac{f-R_if}{2x_i} = \frac{f}{x_i} = g$ and $D_jf=0$ for $j\ne i$,
and Lemma \ref{lem:diriform} gives
\begin{align*}
 \E_\eta[(Af)^2]
&=\cE(f,Af)
=\sum_j\E_\eta[\partial_jf\,\partial_j(Af)]
+\E_\eta\left[c_ig\frac{Af}{X_i}\right]\\
&=\sum_j\cE(\partial_jf)
+\E_\eta[\langle\nabla^2W\nabla f,\nabla f\rangle]
+\sum_{j\ne i}\E_\eta\left[-\frac{c_j'}{X_j}(\partial_jf)^2\right]\\
&\qquad
+\E_\eta\left[
\left(\frac{2c_i}{X_i}-c_i'\right)\partial_if\,\partial_ig
\right]
+\E_\eta\left[c_ig\frac{Af}{X_i}\right].
\end{align*}
To compute the last term, we use $f=x_ig$ so that
\begin{align*}
Af
&=x_i\left(
-\Delta g+\langle\nabla W,\nabla g\rangle
-\sum_k\frac{c_k}{x_k}\partial_kg
\right)
-2\partial_ig+W_ig-\frac{c_i}{x_i}g+\frac{c_i}{x_i}g\\
&=x_iAg-2\partial_ig+W_ig,
\end{align*}
where the last equality uses the unconditionality of $g$.
Moreover, $c_ig$ is unconditional, so Lemma \ref{lem:diriform} gives
\begin{align*}
\E_\eta[c_igAg]
&=\cE(g,c_ig)
=\E_\eta[\langle\nabla g,\nabla(c_ig)\rangle]=\E_\eta[c_i|\nabla g|^2]
+\E_\eta[c_i'g\,\partial_ig].
\end{align*}
Consequently,
$$
\E_\eta\left[c_ig\frac{Af}{X_i}\right]
=\E_\eta[c_i|\nabla g|^2]
+\E_\eta\left[
\left(c_i'-\frac{2c_i}{X_i}\right)g\,\partial_ig
\right]
+\E_\eta\left[c_i\frac{W_i}{X_i}g^2\right].
$$
Finally, since $\partial_if-g=x_i\partial_ig$,
\begin{align*}
\left(\frac{2c_i}{x_i}-c_i'\right)\partial_if\,\partial_ig
+\left(c_i'-\frac{2c_i}{x_i}\right)g\,\partial_ig
=\left(\frac{2c_i}{x_i}-c_i'\right)
(\partial_if-g)\partial_ig
=(2c_i-x_ic_i')(\partial_ig)^2.
\end{align*}
Substituting into the expression for $ \E_\eta[(Af)^2]$ proves
\begin{align*}
 \E_\eta[(Af)^2]
={}&\sum_j\cE(\partial_jf)
+\E_\eta[\langle\nabla^2W\nabla f,\nabla f\rangle]
+\sum_{j\ne i}\E_\eta\left[-\frac{c_j'}{X_j}(\partial_jf)^2\right]\\
&+\E_\eta[c_i|\nabla g|^2]
+\E_\eta[(2c_i-X_ic_i')(\partial_ig)^2]
+\E_\eta\left[c_i\frac{W_i}{X_i}g^2\right].
\end{align*}

\end{proof}

\subsection{Proof of Proposition~\ref{prop:keyestimate}}\label{sec:parity}
\begin{proof}[Proof of Proposition~\ref{prop:keyestimate}]

We conduct the proof in several parts. We first identify the eigenfunction $f$ and show that we can choose it to have the declared symmetries and monotonicity properties. We will then use these properties to establish the lower bound on $\lambda.$

\textbf{Step 1: Choose an $\lambda$-eigenfunction of $A$ that in each coordinate is either odd or even.}
The fact that $\lambda > 0$ and that there is a corresponding eigenfunction $f$ is justified in the appendix in Lemma \ref{lem:closure}. To show that we can choose $f$ to have the appropriate symmetries we effectively use the arguments from the proofs of \cite[Corollary 2(i)] {klartag2009berry} and \cite[Corollary 13]{barthe2020spectral} and adapt them to our setting. The key observation is that for each $j \in[n]$, the reflection operator $R_j, j \in [n]$ commutes with $A$ and the same holds for the projections
\[
 P_j^+=\frac{I+R_j}{2},\qquad P_j^-=\frac{I-R_j}{2}.
\]
Note that $P_j^+f$ is even in the $j^{th}$ coordinate, while $P_j^-f$ is odd and that the family of operators $\{P_j^+,P^-_j\}_{j\in[n]}$ commute with each other.
Consider now the $2^n$ products of these projections, 
$$P_{\varepsilon}=\prod_{j=1}^n P^{\varepsilon_j}_j,\qquad \varepsilon \in \{-,+\}^n,$$
for which we have
$$f = \sum\limits_{\varepsilon}P_{\varepsilon}f.$$
%which decomposes any function $f$ into $2^n$ components $(P_{\varepsilon}(f))_{\varepsilon \in \{-,+\}^n}$ which are orthogonal under $\mathcal{L}^2(\mu)$. 
Since $f\neq 0$ is non-zero, there is at least one $\varepsilon_0$ such that $P_{\varepsilon_0}f \neq 0$. Since  $P_{\varepsilon_0}$ commutes with $A$, $P_{\varepsilon_0}f$ is also a $\lambda$-eigenfunction. 

We may therefore indeed choose a $\lambda$-eigenfunction $f$ of $A$ such that, for each $j \in [n]$, either $R_jf=f$, i.e., it is even with respect to the $j^{th}$ coordinate, or $R_jf=-f$ i.e., it is odd with respect to the $j^{th}$ coordinate. Normalizing the function we may also assume $\E_\eta[ f^2]=1$. Finally, since $\lambda>0$ and constants lie in the kernel of $A$, by self-adjointness of $A$ we have $\E_\eta[f]=0$.

\textbf{Step 2: Rule out two or more odd coordinates.}
Let $f$ be the eigenfunction from the previous step, we will now rule out the possibility that it is odd in more than one coordinate. Towards a contradiction, suppose that $f$ is odd in two distinct coordinates $i$ and $j$, and define the function,
\[
 g=\sgn(x_i)f.
\]
Then $g$ is even in coordinate $i$ and remains odd in coordinate
$j$. In particular, $\E_\eta[g]=0$, and since $g^2 = f^2$ we also have $\E_\eta[g^2] = \E_{\eta}[f^2] = 1.$
Thus, we can arrive at a contradiction by showing $\mathcal{E}(g) < \mathcal{E}(f) = \lambda.$ To estimate $\cE(g)$ we first compute,
$$
\nabla g=\sgn(x_i)\nabla f
\quad \text{ and } \quad
D_i g
=\sgn(x_i)\frac{f+R_i f}{2x_i}
=0,
$$
and, for $k\neq i$,
$$
D_k g
=\frac{\sgn(x_i)f-\sgn(x_i)R_kf}{2x_k}
=\sgn(x_i)D_kf.
$$
Since $D_if=\frac{f}{x_i}$, these identities give
\begin{align*}
\cE(g)
&=\E_\eta\left[|\nabla g|^2\right]
  +\sum_{k=1}^n\E_\eta\left[c_k(D_kg)^2\right] =\E_\eta\left[|\nabla f|^2\right]
  +\sum_{k\neq i}\E_\eta\left[c_k(D_kf)^2\right]\\
&=\cE(f)-\E_\eta\left[c_i\frac{f^2}{x_i^2}\right]
<\lambda.
\end{align*}
The inequality is strict because $c_i>0$ almost everywhere. Thus we have arrived at a contradiction, and can conclude that $f$ cannot be odd in more than a single coordinate.

\textbf{Step 3: Rule out all even coordinates.}
Suppose next that $f$ is even in every coordinate. In other words, $D_jf = 0$ for every $j \in [n]$, which, as in \eqref{eq:reflection-form} shows
$$\E_\eta\left[|\nabla f|^2\right] = \cE(f) = \lambda.$$
Furthermore, by definition of $\lambda$, we have $\cE(\partial_jf)\geq
\lambda\E_\eta[(\partial_jf)^2]$, which together with \eqref{eq:evenidentity} shows,
\begin{align*}
0=\norm{Af}_2^2-\lambda\E_\eta|\nabla f|^2\geq
\E_\eta\left[\ip{\nabla^2W\nabla f}{\nabla f}\right]
+\sum_{j=1}^n
\E_\eta\left[-\frac{c_j'}{x_j}(\partial_jf)^2\right].
\end{align*}
Recall that $q(r)=\left(\frac{r}{\sqrt{S^2+r^2}}\right)^{1/2}$, and as in \eqref{eq:qchoices}
$$
-\frac{c_j'}{x_j}
=\frac{S^2}{(S^2+x_j^2)^2}.
$$
Consequently, since
$\nabla^2W\succeq0$, and $f$ is not constant,
$$
0\geq
\sum_{j=1}^n
\E_\eta\left[
\frac{S^2}{(S^2+x_j^2)^2}(\partial_jf)^2
\right]>0,
$$
a contradiction. Thus $f$ is odd in exactly one coordinate, which we henceforth denote by $i$, and even in all the others. 

\textbf{Step 4: Choosing $f$ so that $f\cdot x_i\geq0$.}
For our eigenfunction $f$, define, in a similar fashion to before,
\[
 g=\sgn(x_i)|f|.
\]
Note that $g$ is odd only in coordinate $i$ and even in the other coordinates. So, the exact same arguments in Step $2$ show that
\[
 \E_\eta[g]=0,\qquad \E_\eta[g^2] = \E_\eta[f^2]=1,\qquad
 \lambda\le\cE(g)\leq\cE(f)=\lambda.
\]
Thus $g$ is also a $\lambda$-eigenfunction and $x_ig = |x_i|\cdot|f|\geq0,$ and we now can replace $f$ with $g$. This concludes the proof of the first item in Proposition \ref{prop:keyestimate}. The next two parts are concerned with the second item.

\textbf{Step 5: An integrated monotonicity property.}
We now prove that $\E_\eta[\partial_i f]\geq0$.
Since $Af=\lambda f$, we have,
\begin{equation} \label{eq:forfinal}
\lambda\E_\eta[X_i f] = \E_\eta[X_i Af]
=\cE(f,x_i)
=\E_\eta[\partial_i f]
+\E_\eta\left[c_i\frac f{X_i}\right],
\end{equation}
where we used $D_ix_i = 1$, and $D_if = \frac{f}{x_i}$, as well as $D_jx_i = D_jf = 0$ when $j \neq i$.
Since $\frac{f}{x_i}\geq0$, it remains to prove
\begin{equation} \label{eq:toshow2}
    \E_\eta\left[c_i\frac f{X_i}\right]
\leq\lambda\E_\eta[X_i f].
\end{equation}
For $t>0$, set
$$
\phi_t(x)=\frac{x_i}{\sqrt{x_i^2+t^2}}.
$$
Substituting in \eqref{eq:DunkelRegu},
we obtain
\begin{equation}\label{eq:smoothsign}
A\phi_t
=\frac{3t^2x_i}{(x_i^2+t^2)^{5/2}}
+\frac{t^2W_i}{(x_i^2+t^2)^{3/2}}
+\frac{c_ix_i}{(x_i^2+t^2)^{3/2}}.
\end{equation}
Since $W$ is convex and unconditional,
$x_iW_i(x) \geq 0$. Together with $x_if \geq 0$,
this implies $fW_i\geq0$. So for $x_i\neq0$,
$
\frac{W_i(x)}{x_i}
\geq0.
$
 Therefore, since $A$ is self-adjoint,
$$
\lambda\E_\eta[f\phi_t]
=\E_\eta[fA\phi_t]
\geq
\E_\eta\left[
\frac{c_i|X_i||f|}{(X_i^2+t^2)^{3/2}}
\right].
$$
Taking $t\to 0$, and applying the dominated convergence theorem on the left side, and the monotone convergence theorem on the right side yields
\begin{equation}\label{eq:signinverse}
\E_\eta\left[\frac{c_i|f|}{X_i^2}\right]
\leq\lambda\E_\eta\left[|f|\right].
\end{equation}
One can verify that for our function $\we$, the function $r \to \frac{\we(r)}{r^2}$
is decreasing. Hence
$$
(|x_i|-|y_i|)
\left(
\frac{c_i(x)}{x_i^2}-\frac{c_i(y)}{y_i^2}
\right)\leq0 \implies \frac{c_i(x)}{|x_i|} + \frac{c_i(y)}{|y_i|} \leq \frac{|y_i|c_i(x)}{x_i^2} + \frac{|x_i|c_i(y)}{y_i^2}.
$$
We now use this inequality to show \eqref{eq:toshow2}. Indeed,
\begin{align*}
&2\E_\eta\left[|f|\right]\,
  \E_\eta\left[c_i\frac f{X_i}\right] = 2\int |f(y)|d\eta(y)\int c_i(x)\frac{|f(x)|}{|x_i|}d\eta(x)\\
  &=2\iint |f(x)f(y)|\frac{c_i(x)}{|x_i|}d\eta(x)d\eta(y)\\
&=\iint |f(x)f(y)|
\left(\frac{c_i(x)}{|x_i|}
     +\frac{c_i(y)}{|y_i|}\right)
\,d\eta(x)\,d\eta(y)\\
&\leq\iint |f(x)f(y)|
\left(\frac{|y_i|c_i(x)}{x_i^2}
     +\frac{|x_i|c_i(y)}{y_i^2}\right)
\,d\eta(x)\,d\eta(y)\\
&=2
  \E_\eta\left[\frac{c_i|f|}{X_i^2}\right]\E_\eta[X_i f]\leq2\lambda\E_\eta\left[|f|\right]\E_\eta[X_i f],
\end{align*} 
where the last inequality is \eqref{eq:signinverse}.
Dividing by $2\E_\eta\left[|f|\right]$ shows \eqref{eq:toshow2} from which we deduce $\E_\eta[\partial_if]\geq 0$.

\textbf{Step 6: Apply the spectral gap to the derivatives.}
For $j\neq i$, the derivative $\partial_jf$ is odd in
coordinates $i$ and $j$. Consequently,
$\sgn(x_i)\partial_jf$ has mean zero.
The calculation from Step 2 gives
\begin{align*}
\lambda\E_\eta[(\partial_jf)^2]
\leq\cE\left(\sgn(x_i)\partial_jf\right)=\cE(\partial_jf)
-\E_\eta\left[
c_i\left(\frac{\partial_jf}{X_i}\right)^2
\right] =\cE(\partial_jf)
-\E_\eta\left[
c_i\left(\partial_j\left(\frac f{x_i}\right)\right)^2
\right].
\end{align*}
Additionally, for the derivative in the $i^{th}$ coordinate, by definition of $\lambda$, we have
$$
\cE(\partial_if)
\geq\lambda\left(\E_\eta[(\partial_if)^2]
-\E_\eta[\partial_if]^2\right).
$$
We would like to plug this inequality into \eqref{eq:oddidentity}. To set the stage for this, we first note that, for our choice of $c_i$, it is straightforward to verify that $2c_i - x_ic_i' \geq c_i$, and hence,
$$\E_\eta\left[
(2c_i-x_ic_i')\left(\partial_i\left(\frac f{x_i}\right)\right)^2
\right]+\sum\limits_{j\neq i}\E_\eta\left[
c_i\left(\partial_j\left(\frac f{x_i}\right)\right)^2
\right]\geq \E_\eta\left[
c_i\left|\nabla\frac f{x_i}\right|^2
\right].$$
Combining the above inequalities into \eqref{eq:oddidentity}, and using $\nabla^2W\succeq0$ as well as $-\frac{c_j'}{x_j}\geq0$,
we obtain
\begin{align*}
\lambda^2
\geq{}&
\lambda\E_\eta\left[|\nabla f|^2\right]
-\lambda\bigl(\E_\eta[\partial_if]\bigr)^2
+2\E_\eta\left[c_i \left|\nabla \left(\frac{f}{x_i}\right)\right|^2\right]+\E_\eta\left[c_i\frac{W_i\,f^2}{X_i^3}\right].
\end{align*}
Since we also have
$$
\lambda=\cE(f)
=\E_\eta\left[|\nabla f|^2\right]
+\E_\eta\left[c_i\frac{f^2}{X_i^2}\right],
$$
it follows that
\begin{equation}\label{eq:doubleenergy}
\lambda\E_\eta[\partial_if]^2
+\lambda\E_\eta\left[c_i\frac{f^2}{X_i^2}\right] \geq 2\E_\eta\left[
c_i\left|\nabla(\frac f{x_i})\right|^2
\right]
+\E_\eta\left[c_i\frac{W_if^2}{X_i^3}\right].
\end{equation}
Finally, we apply Lemma \ref{lem:weighted} to $\frac{f}{x_i}$, which is even in all coordinates. Since we aim to prove \eqref{eq:Jupper} we may assume that $\lambda \geq \frac{3}{S^2} \implies \lambda-\frac{3}{2S^2}\geq\frac{\lambda}{2}$, and so
$$
2\mathbb E_\eta\left[
c_i\left|\nabla\!\left(\frac f{x_i}\right)\right|^2
\right]
\geq
\lambda\mathbb E_\eta\left[c_i\frac{f^2}{X_i^2}\right]
-\frac{\lambda}{\mathbb E_\eta[c_i]}
\mathbb E_\eta\left[c_i\frac f{X_i}\right]^2. 
$$
Substituting into \eqref{eq:doubleenergy}, we obtain
$$
\lambda\mathbb E_\eta[\partial_if]^2
+\frac{\lambda}{\mathbb E_\eta[c_i]}
\mathbb E_\eta\left[c_i\frac f{X_i}\right]^2\geq \mathbb E_\eta\left[c_i\frac{W_if^2}{X_i^3}\right].
$$
Since \(0<\mathbb E_\eta[c_i]\leq1/2\), we also have
$$
\begin{aligned}
\mathbb E_\eta\left[c_i\frac{W_i\,f^2}{X_i^3}\right]
&\leq
\frac{\lambda}{\mathbb E_\eta[c_i]}
\left(
\mathbb E_\eta[\partial_if]^2
+\mathbb E_\eta\left[c_i\frac f{X_i}\right]^2
\right)\\
&\leq
\frac{\lambda}{\mathbb E_\eta[c_i]}
\left(
\mathbb E_\eta[\partial_if]
+\mathbb E_\eta\left[c_i\frac f{X_i}\right]
\right)^2=
\frac{\lambda^3\bigl(\mathbb E_\eta[X_if]\bigr)^2}
{\mathbb E_\eta[c_i]}.
\end{aligned}
$$
The second inequality follows from $\E_\eta[\partial_if],\E_\eta[c_i\frac{f}{X_i}]\geq 0$, established in the previous steps, and the final equality is \eqref{eq:forfinal}. Rearranging terms then leads to \eqref{eq:Jupper} which finishes the proof.
\end{proof}

\appendix
\section{Operator domains and approximation}\label{sec:analyticdetails}

For the auxiliary measure in Theorem~\ref{thm:bounded}, we assume that
$W\in C^2(\mathbb R^n)$ is unconditional, and that it satisfies \eqref{eq:operator-assumptions}. I.e., for some
$K,\kappa_0>0$ and $B_0\ge0$,
\begin{equation*}
 0\preceq\nabla^2W\preceq KI,
 \qquad x\cdot\nabla W(x)\ge\kappa_0|x|^2-B_0.
\end{equation*}
We immediately note the following consequences of \eqref{eq:operator-assumptions}:
First, if $\theta \in \R^n$ is a unit vector, we can integrate the second inequality along $t\to t\theta$, and obtain for $r>1$
\begin{equation} \label{eq:subgaus}
    W(r\theta)\ge W(\theta)+\frac{\kappa_0}{2}(r^2-1)-B_0\log r.
\end{equation}
Consequently $W(x)\ge\kappa_0|x|^2/4-b$ for some finite $b$,
so $\eta$ has Gaussian tails. Unconditionality also gives
\begin{equation}\label{eq:nablaWbound}
     \frac{W_j(x)}{x_j}
 =\int_0^1 W_{jj}(x_1,\ldots,tx_j,\ldots,x_n)\,dt\in[0,K],
 \qquad |\nabla W(x)|\le K|x|.
\end{equation}

We will use these properties to establish that $A$ is a self-adjoint operator on its domain. We use $q,\we$ for the functions defined in \eqref{eq:qchoices} and $A$ is the Dunkl Langevin operator from \eqref{eq:DunkelRegu}, with $\cE$ the Dirichlet form, as in Lemma \ref{lem:diriform}. 

\subsection{The self-adjoint operator}\label{sec:operator-domain}
Our first order of business is to show that $A$ is a self-adjoint operator on its domain and that it has a well-defined spectral gap. Such results are available in the literature, mostly applying to the Dunkl Laplacian. Since those results do not seem to apply directly to the specific operator that we've defined, the proof goes by reducing our operator $A$ to the Dunkl Laplacian.

\begin{lemma} \label{lem:closure}
    Let $A$ be as in \eqref{eq:DunkelRegu} and suppose that $W$ satisfies \eqref{eq:operator-assumptions}. Then, the closure of $A$ on $C_c^{\infty}(\R^n)$ in $L^2(\eta)$ is a non-negative self-adjoint operator, which we again denote by $A$. Moreover, $A$ has a discrete spectrum and 
    $$\mathrm{ker}(A) = \{\textnormal{constant functions}\}.$$
\end{lemma}

\begin{proof}
We first check that $A$ is well-defined for $f \in C_c^{\infty}(\R^n).$   
We first observe that Taylor's formula with an integral remainder gives for $x_j\ne0$,
\begin{align} \label{eq:taylor}
    \frac{\partial_jf(x)}{x_j}
-\frac{f(x)-f(R_jx)}{2x_j^2}
=
\frac12\int_{-1}^1(1+t)
\partial_{jj}f(x_1,\ldots,tx_j,\ldots,x_n)\,dt.
\end{align}
Since $0<c_j\le1/2$, it follows that
\begin{align*}
    \left|
-\frac{c_j(x)}{x_j}\partial_jf(x)
+\frac{c_j(x)}{2x_j^2}\bigl(f(x)-f(R_jx)\bigr)
\right|
\le\frac12\|\partial_{jj}f\|_\infty.
\end{align*}
Since $f$ is compactly supported and $W$ is locally-Lipschitz both terms $-\Delta f$ and $\langle \nabla W,\nabla f\rangle$ in $Af$, as in \eqref{eq:DunkelRegu}, are bounded, and therefore, $Af\in L^2(\eta)$.

Next, to understand the closure of $A$, we write 
$$
w(x)=\prod_j|x_j|^{1/2},
\qquad
B(x)=W(x)+\frac14\sum_j\log(S^2+x_j^2),
$$
so that $d\eta(x)=Z^{-1}e^{-B(x)}w(x)dx$, for a normalizing constant $Z>0$. Let
$$
\Delta_Dh=\Delta h+\frac12\sum_j
\left(\frac{\partial_jh}{x_j}-\frac{h-R_jh}{2x_j^2}\right)
$$
be the Dunkl Laplacian.
The map $Uf=Z^{-1/2}e^{-B/2}f$ is unitary from $L^2(\eta)$
to $L^2(w\,dx)$. Substituting
$\partial_jB=\partial_jW+\frac{x_j}{2(S^2+x_j^2)}$ gives
\begin{align*}
A=-\Delta_D+\langle\nabla B,\nabla\rangle-\mathcal K,
\quad \text{ where }\quad \mathcal K=\sum_j\frac{\mathrm{I}-R_j}{4(S^2+x_j^2)} = \sum\limits_j \frac{\frac{1}{2}-c_j(x)}{2x_j^2}\left(\mathrm{I}-R_j\right).
\end{align*}
Note that $\mathcal K$ is bounded and self-adjoint.
Now, since $B$ is unconditional and $R_jB = B$ for every $j\in [n]$, standard computations show for a test function $h$,
\begin{align*}
    e^{-B/2}\Delta_D(e^{B/2}h) = \Delta_Dh + \langle \nabla B,\nabla h\rangle + \left(\frac{1}{2}\Delta_DB + \frac{1}{4}|\nabla B|^2\right)h.
\end{align*}
Similarly,
\begin{align*}
    e^{-B/2}\langle\nabla B,\nabla(e^{B/2}h)\rangle &= \langle \nabla B, \nabla h\rangle + \frac{1}{2}|\nabla B|^2h\\
    \mathcal{K}(e^{B/2}h) &=   e^{B/2}\mathcal{K}(h).
\end{align*}
Taking $e^{B/2}h = f$, we see
\begin{align} \label{eq:UAcommute}
UAf=(-\Delta_D+Q-\mathcal K)Uf,
\quad \text{ where }\quad 
Q=\frac14|\nabla B|^2-\frac12\Delta_DB.
\end{align}
Since we assumed that $W$ is unconditional and has a bounded Hessian, the same is also true for $B$, and we have
$$
\frac{\partial_jB(x)}{x_j}
=\int_0^1\partial_{jj}B(x_1,\ldots,tx_j,\ldots,x_n)\,dt.
$$
In particular, $\Delta_DB = \Delta B + \frac{1}{2}\sum\limits_j \frac{\partial_j B}{x_j}$ is bounded and continuous.
Moreover, because of \eqref{eq:operator-assumptions},
$$
|x||\nabla B(x)|\geq \langle x,\nabla B(x)\rangle
=\langle x,\nabla W(x)\rangle
+\sum_j\frac{x_j^2}{2(S^2+x_j^2)}
\ge\kappa_0|x|^2-B_0.
$$
Thus $|\nabla B(x)|\ge\kappa_0|x|/2$ for sufficiently
large $|x|$. Recalling the definition of $Q$, we obtain
\begin{equation} \label{eq:lowerQ}
  Q(x)\ge a|x|^2-b  
\end{equation}
for some $a>0$ and $b\ge0$.
We can now see that $Q+b$ is non-negative and locally bounded, and we can apply \cite[Theorem~4.6]{amri2019dunkl}(see also \cite[Section 2.3]{hejna2021schrodinger}) to conclude that
$-\Delta_D+Q$ is essentially self-adjoint on $C_c^\infty(\R^n)$.
The same is true for $-\Delta_D+Q-\mathcal K$, since
$\mathcal K$ is bounded and self-adjoint.
We denote its closure by $H$ and the domain by $D(H) \subset L^2(wdx)$.

We shall use $H$ to identify the closure of $A$. Repeating the computation from \eqref{eq:taylor},
$$
\frac{\partial_jh(x)}{x_j}
-\frac{h(x)-h(R_jx)}{2x_j^2}
=\frac{1}{2}\int_{-1}^1(1+t)
\partial_{jj}h(x_1,\ldots,tx_j,\ldots,x_n)\,dt.
$$
Consequently, if $h_m\to h$ in $C^2$ and all these functions
are supported in a common compact set, then
$$
(-\Delta_D+Q-\mathcal K)h_m
\longrightarrow
(-\Delta_D+Q-\mathcal K)h
\quad\text{in }L^2(w\,dx).
$$
Approximating by smooth functions and using the closedness of $H$,
we see that $C_c^2(\R^n)\subset D(H)$.
Now, for $\psi\in C_c^\infty(\R^n)$, take $f_m\in C_c^\infty(\R^n)$
converging to $U^{-1}\psi$ in $C^2(\R^n)$, with supports in a common compact set. The above observation gives
$$
Uf_m\longrightarrow\psi,
\qquad HUf_m\longrightarrow H\psi
\quad\text{in }L^2(w\,dx).
$$
Since every element of $D(H)$ can be approximated, together
with its image under $H$, by functions in $C_c^\infty(\R^n)$,
the same is true using functions in $U\left(C_c^\infty(\R^n)\right)$.
The identity $UAf=HUf$ from \eqref{eq:UAcommute} therefore shows that the closure of
$A$ is $U^{-1}HU$. We again denote this closure by $A$ and the domain $D(A)\subset L^2(\eta)$.
Furthermore, Lemma~\ref{lem:diriform} gives
$$
\E_{wdx}\left[HUf\cdot Uf\right] 
=\E_\eta\left[ Af\cdot f\right]
=\cE(f)\ge0,
$$
for $f\in C_c^\infty(\R^n)$. Passing to the limit proves that
$H$ and $A$ are non-negative.

To show that the spectrum is discrete, let $(e_m)$ be an
orthonormal eigenbasis of $-\Delta_D+a|x|^2$, with eigenvalues
$\mu_m\to\infty$, arranged in increasing order.
Such a basis exists by \cite[Corollary~3.5(ii)]{rosler1998hermite}, with an appropriate change of variables. From here, we will slightly abuse notation and use $\langle\cdot,\cdot\rangle$ for the inner product in $L^2(wdx)$. For $u\in C_c^\infty(\R^n)$, the bound in \eqref{eq:lowerQ} gives,
$$
\sum_m\mu_m|\langle u,e_m\rangle|^2
=\langle(-\Delta_D+a|x|^2)u,u\rangle
\le\langle Hu,u\rangle+C\|u\|_2^2,
\qquad C=b+\|\mathcal K\|.
$$
An approximation argument, as above, gives the same result on $D(H)$. Take now $u=(H+\mathrm{I})^{-1}v$ for some function $v$ with
$\langle v,v\rangle\leq1$.
Since $H$ is non-negative,
$$
\langle u,u\rangle
\le \langle Hu,u\rangle+\langle u,u\rangle
=\langle v,u\rangle
\leq \sqrt{\langle u,u\rangle}.
$$
Thus $\langle u,u\rangle\leq1$ and $\langle Hu,u\rangle\leq1$.
As the eigenvalues $\mu_m$ are arranged in increasing order,
$$
\mu_{N+1}\sum_{m>N}|\langle u,e_m\rangle|^2
\le \sum_{m>N}\mu_m|\langle u,e_m\rangle|^2
\le 1+C.
$$
If $P_N$ is the orthogonal projection onto
$\operatorname{span}\{e_1,\ldots,e_N\}$, we can equivalently state
$$
\|(H+I)^{-1}-P_N(H+I)^{-1}\|_{\mathrm{op}}
\le \sqrt{\frac{1+C}{\mu_{N+1}}}
\longrightarrow0.
$$
Hence $(H+I)^{-1}$ is a uniform limit of finite-rank
operators and is therefore compact \cite[Theorem 4.18]{rudin1991functional}. The spectral theorem for compact operators now implies that $H$ has a discrete spectrum, and the same is also true for $A$ since it is unitarily equivalent to $H$.

Finally, to understand the kernel of $A$, take $B_1 \subseteq B_2$ two centered Euclidean balls and take $\chi$ to a smooth radially symmetric function which is $1$
on $B_1$ and $0$ outside $B_2$, and for $R>0$ set $\chi_R(x)=\chi(x/R)$.
Since $|\nabla W(x)|\le K|x|$, the definition of $A$ gives
\begin{equation} \label{eq:smoothcut}
    \|\chi_R-1\|_{L^2(\eta)}\longrightarrow0,
\qquad
\|A\chi_R\|_{L^2(\eta)}^2
\le C\cdot\eta(RB_{2}\setminus RB_1)\leq C\cdot\eta(\R^n\setminus RB_1) \longrightarrow0.
\end{equation}

Hence $1\in D(A)$ and it is immediate that $A1=0$.
Conversely, if $f\in\ker(A)$, take $f_m\in C_c^\infty(\R^n)$
such that $f_m\to f$ and $Af_m\to0$ in $L^2(\eta)$.
By Lemma \ref{lem:diriform},
$$
\E_\eta\left[|\nabla f_m|^2\right]
+\sum_j\E_\eta\big[c_j(D_jf_m)^2\big]
=\cE(f_m) = \E_\eta[f_m\cdot Af_m]\longrightarrow0.
$$
If $K$ is a compact subset of an orthant in $\R^n$, then the density of $\eta$ is bounded from below on $K$. Passing to the limit gives $\nabla f\equiv 0$ on $K$. Since we can take a compact exhaustion of each orthant, we also get that $f$ is constant on each orthant. 

It remains to show that $f$ is constant everywhere. For every coordinate $j$, and $K$ as above, we have
$$
\int_K|f_m-R_jf_m|^2d\eta
\le
\left(\sup_{x\in K}\frac{4x_j^2}{c_j(x)}\right)
\E_\eta[c_j(D_jf_m)^2]
\longrightarrow0.
$$
Since $\eta$ is unconditional, $R_jf_m\to R_jf$ in
$L^2(\eta)$. Passing to the limit gives
$$
\int_K|f-R_jf|^2\,d\eta=0.
$$
Hence $f\mid_{K} \equiv f\mid_{R_jK}$, which implies that $f$ is constant on the union of orthants containing $K \cup R_jK$. Since $j$ is arbitrary we conclude that $f$ is constant everywhere.
\end{proof}
\subsection{Functions in the domain}
Next we will establish that the functions we work with are indeed in the domain of $A$ and that our formulas extend to these functions.
\begin{lemma}\label{lem:domains}
Under the assumptions of Lemma~\ref{lem:closure}, the following hold:
\begin{enumerate}
\item Let $f$ be a locally-Lipschitz function satisfying
$$
\E_\eta[f^2+|\nabla f|^2]<\infty.
$$
Then $f\in D(A^{1/2})$, and
$$\cE(f)
\le 9\E_\eta[|\nabla f|^2].
$$
\item If in addition, $f\in C^2(\R^n)$ and
$f,\nabla f,\nabla^2f$ have polynomial growth, then
$f\in D(A)$, and $Af$ is given by \eqref{eq:DunkelRegu}.
\end{enumerate}
\end{lemma}

\begin{proof}
For $R > 0$, let $\chi_R$ be a smooth cutoff, as in \eqref{eq:smoothcut}. If $\rho$ is some compactly supported smooth mollifier and $\rho_m(x) = m^n\rho(x\cdot m)$ we consider the mollified function $f_m = \chi_m f *\rho_m$. Then $f_m \in C_c^\infty(\R^n)$ and
$$
\E_\eta\left[|f_m-f|^2\right]+\E_\eta\left[|\nabla f_m-\nabla f|^2\right]
\longrightarrow0,
$$
By Lemmas~\ref{lem:diriform} and~\ref{lem:diriforms}
$$
\E_\eta\left[|A^{1/2}(f_m-f_\ell)|^2\right]
=\cE(f_m-f_\ell)
\leq9\E_\eta\left[|\nabla f_m-\nabla f_\ell|^2\right].
$$
Thus, $A^{1/2}f_m$ is a Cauchy sequence and $A^{1/2}f_m$ converges to some $g\in L^2(\eta)$.
Since $f_m\to f$ in $L^2(\eta)$ and $A^{1/2}$ is closed,
we obtain $f\in D(A^{1/2})$ and $A^{1/2}f=g$, i.e. $f\in D(A^{1/2})$ and $$\E_\eta[|A^{1/2}(f_m-f)|^2]\to 0.$$ 
To extend Lemma \ref{lem:diriforms} to $f$, we also observe that for every $j\in[n]$,
$$\E_\eta \left[c_j(D_jf_m-D_jf_\ell)^2\right] \leq \cE(f_m-f_\ell).$$
So, $\sqrt{c_j}D_jf_m$ is also Cauchy. Since $f_m\to f$ in $L^2(\eta)$ and $\eta$ is unconditional,
$R_jf_m\to R_jf$ in $L^2(\eta)$ as well.
Passing to a subsequence gives $D_jf_m\to D_jf$,
$\eta$-almost everywhere.
Thus the $L^2(\eta)$ limit of the Cauchy sequence
$\sqrt{c_j}D_jf_m$ is $\sqrt{c_j}D_jf$. Combining all of the above we get
$$\cE(f) = \E_\eta\left[|A^{1/2}f|^2\right] =  \lim\limits_{m\to \infty}\E_\eta\left[|A^{1/2}f_m|^2\right] \leq 9\lim\limits_{m\to \infty}\E_\eta[|\nabla f_m|^2] = 9\E_\eta[|\nabla f|^2].$$

For the second claim, we first note in passing that the arguments in the proof of Lemma \ref{lem:closure} show that $C_c^2(\R^n) \subset D(A)$. Now, suppose that $f\in C^2(\R^n)$
and that $f,\nabla f,\nabla^2f$ have polynomial growth.
Let us treat for now $Af$ according to the definition in \eqref{eq:DunkelRegu}, and show that $Af \in L^2(\eta)$. Then we will show that for an appropriate sequence in $f_m \in C_c^2(\R^n)$, $Af_m$ converges to $Af$ in $L^2(\eta)$, which will conclude the proof. 

Recalling first that $\eta$ has sub-Gaussian tails, as in \eqref{eq:subgaus}, it has moments of every order, and so $f,\nabla f,\nabla^2 f$ are square-integrable. The Taylor formula in \eqref{eq:taylor}
$$
\frac{\partial_jf(x)}{x_j}
-\frac{f(x)-f(R_jx)}{2x_j^2}
=
\frac12\int_{-1}^1(1+t)
\partial_{jj}f(x_1,\ldots,tx_j,\ldots,x_n)\,dt
$$
and the bound $|\nabla W(x)|\le K|x|$ from \eqref{eq:nablaWbound} show that $Af$
also has polynomial growth. Consequently,
$\ Af\in L^2(\eta)
$ as well.

To show the necessary approximation, let $\chi_R$ be as above. Since $\chi_R$ is radial, $R_j\chi_R=\chi_R$, hence
$$
\chi_Rf-R_j(\chi_Rf)=\chi_R(f-R_jf).
$$
Substituting into the definition of $A$ and expanding
the derivatives gives
\begin{align*}
A(\chi_Rf)
={}&\chi_RAf
-2\langle\nabla\chi_R,\nabla f\rangle +f\left(
-\Delta\chi_R+\langle\nabla W,\nabla\chi_R\rangle
-\sum_j\frac{c_j}{x_j}\partial_j\chi_R
\right).
\end{align*}
The last term is precisely $fA\chi_R$, since
$R_j\chi_R=\chi_R$. Therefore,
$$
A(\chi_Rf)
=\chi_RAf+fA\chi_R
-2\langle\nabla\chi_R,\nabla f\rangle.
$$
Moreover the radial symmetry also gives $\left|\frac{\partial_j\chi_R}{x_j}\right|\le \frac{C}{R^2}$, for some constant $C>0$.
Together with $|\nabla W(x)|\le K|x|$, this yields, for $R\ge1$,
$$
|A\chi_R|\le C\mathbf1_{RB_2\setminus RB_1},
\qquad
|\nabla\chi_R|\le\frac CR.
$$
Therefore,
\begin{align*}
\E_\eta\left[|A(\chi_Rf)-Af|^2\right]
&\leq 8\left(\E_\eta\left[|(\chi_R-1)Af|^2\right]
+C^2\E_\eta\left[\mathbf1_{\R^n\setminus RB_1}f^2\right]
+\frac {C^2}{R^2}\E_\eta\left[|\nabla f|^2\right]\right)
&\xrightarrow{R\to \infty}0.
\end{align*}
The first two terms vanish by dominated convergence.
Since $\chi_Rf\in C_c^2(\R^n)\subset D(A)$ and
$\chi_Rf\to f$ in $L^2(\eta)$, this proves that
$f\in D(A)$, with $Af$ given by \eqref{eq:DunkelRegu}.
\end{proof}
In light of Lemma \ref{lem:domains} and its proof, we emphasize the fact that the same approximation argument shows that if $f\in D(A^{1/2})$ then,
$$
\E_\eta[|A^{1/2}f|^2] = \cE(f)
=\E_\eta[|\nabla f|^2]
+\sum_j\E_\eta[c_j(D_jf)^2],
$$
and $f$ is weakly differentiable.
Thus the first identity of Lemma \ref{lem:diriform} holds for
$f\in D(A)$ and $g\in D(A^{1/2})$. %Lemma \ref{lem:closure} shows this is also true when $f\in D(A)$.

We next prove that certain operations preserve the domain.
\begin{lemma}\label{lem:form-signs}
Let $f \in D(A^{1/2})$. Then, if $R_if=-f$ for some $i\in[n]$, 
$$\operatorname{sgn}(x_i)f\in D(A^{1/2}) \quad \text{ and } \quad \operatorname{sgn}(x_i)|f| \in D(A^{1/2}).$$
Furthermore
\begin{align*}
\cE(\operatorname{sgn}(x_i)f)
&=\cE(f)-\E_\eta[c_i(D_if)^2],\\
\cE(\operatorname{sgn}(x_i)|f|)
&\le\cE(f).
\end{align*}
\end{lemma}

\begin{proof}
Fix $i \in [n]$ and suppose that $R_if=-f$.
Since $AR_i=R_iA$ holds on $C_c^\infty$ it extends
to $D(A)$ by approximation.
Thus $R_i$ commutes
with $A^{1/2}$ as well. Choose $f_m\in C_c^\infty(\R^n)$ with
$$
f_m\longrightarrow f,\qquad
A^{1/2}f_m\longrightarrow A^{1/2}f
\quad\text{in }L^2(\eta).
$$
Replacing $f_m$ by $(f_m-R_if_m)/2$ preserves the limits and ensures that $R_if_m=-f_m$.  Write now
$$
h_m=\operatorname{sgn}(x_i)f_m,
\qquad
v_m=\operatorname{sgn}(x_i)|f_m|
$$
which are compactly supported and locally Lipschitz, so
Lemma \ref{lem:domains} applies. 

We first handle the sequence $h_m$, for which we have, almost everywhere,
$$
\nabla h_m=\operatorname{sgn}(x_i)\nabla f_m,\qquad
D_ih_m=0,\qquad
D_jh_m=\operatorname{sgn}(x_i)D_jf_m \text{ when }j\neq i.
$$
Consequently,
$$
\cE(h_m)=\cE(f_m)-\E_\eta[c_i(D_if_m)^2].
$$
Applying the same calculation to $f_m-f_\ell$ gives
$$
\|A^{1/2}(h_m-h_\ell)\|_{L^2(\eta)}^2
\le\cE(f_m-f_\ell)\longrightarrow0.
$$
Since $h_m\to h:=\operatorname{sgn}(x_i)f$ in $L^2(\eta)$,
we get $h\in D(A^{1/2})$ and
$A^{1/2}h_m\to A^{1/2}h$ in $L^2(\eta)$.
Similarly, $\sqrt{c_i}D_if_m\to\sqrt{c_i}D_if$ in $L^2(\eta)$.
Passing to the limit now gives
$$
\cE(h)=\cE(f)-\E_\eta[c_i(D_if)^2].
$$
Repeating the arguments for the sequence $v_m$, we have almost everywhere,
$$
|\nabla v_m|\le|\nabla f_m|,\qquad
|D_iv_m|=|D_if_m|,\qquad
|D_jv_m|\le|D_jf_m|\text{ when } j\ne i.
$$
Thus $\cE(v_m)\le\cE(f_m)$, and
$v_m\to v:=\operatorname{sgn}(x_i)|f|$ in $L^2(\eta)$. For every $\psi\in D(A^{1/2})$, the convergence
$v_m\to v$ in $L^2(\eta)$ gives
\begin{align*}
\left|\E_\eta[vA^{1/2}\psi]\right|
=\lim_{m\to\infty}
\left|\E_\eta[(A^{1/2}v_m)\psi]\right|\leq\lim_{m\to\infty}\sqrt{\cE(f_m)\E_\eta\left[\psi^2\right]}
=\sqrt{\cE(f)\E_\eta\left[\psi^2\right]}.
\end{align*}
This proves that
$v\in D((A^{1/2})^*)$,
and since $A^{1/2}$ is self-adjoint, we obtain
$v\in D(A^{1/2})$. The same bound also gives
$$
\cE(v)\leq\cE(f).
$$
\end{proof}
Finally, we deal with functions in $D(A)$ which have some symmetries.
\begin{lemma}\label{lem:derivative-domains}
Let $f\in D(A)$ be either even in every coordinate, or
odd in one coordinate $x_i$ and even in all the others.
Then:
\begin{enumerate}
    \item For every $j\in [n]$, $
\partial_jf\in D(A^{1/2})
$ and the corresponding identity in Lemma \ref{lem:identities} holds.
\item In the second case, $g=\frac{f}{x_i}$ is weakly differentiable in each open orthant, $
\E_\eta[c_i(g^2+|\nabla g|^2)]<\infty,$
and the inequality of Lemma \ref{lem:weighted} holds for $g$.
\end{enumerate}
\end{lemma}
\begin{proof}
As before, we choose an approximating sequence $f_m
\in C_c^\infty(\R^n)$ such that
$$
\E_\eta\left[|f_m-f|^2\right]+\E_\eta\left[|Af_m-Af|^2\right]\longrightarrow0.
$$
As in Lemma \ref{lem:form-signs} since the reflections $R_j$ commute with $A$, we can assume $f_m$ has the same symmetries as $f$, and moreover,
$$
\cE(f_m-f) 
\leq\sqrt{\E_\eta\left[|f_m-f|^2\right]\E_\eta\left[|Af_m-Af|^2\right]}
\longrightarrow0.
$$
In particular, since $\E_\eta[(\partial_jf_m - \partial_jf)^2]\leq \cE(f_m-f)$, we get $\partial_jf_m\to\partial_jf$ in $L^2(\eta)$.

Every term on the right-hand side of either identity in
Lemma \ref{lem:identities} is non-negative. Applying Lemma \ref{lem:identities} to $f_m-f_\ell$ therefore gives
$$
\cE(\partial_jf_m - \partial_jf_\ell)=\E_\eta\left[|A^{1/2}(\partial_jf_m-\partial_jf_\ell)|^2\right]
\leq \E_\eta\left[|A(f_m-f_\ell)|^2\right]
\longrightarrow0.
$$
Closedness of $A^{1/2}$ now gives
$$
\partial_jf\in D(A^{1/2}) \quad \text{ and }\quad
A^{1/2}\partial_jf_m\longrightarrow A^{1/2}\partial_jf
\quad\text{in }L^2(\eta).
$$
This concludes the first part. For the second part, put $g_m=\frac{f_m}{x_i}$, with derivatives
taken on the open orthants. Since $g_m=D_if_m$ and $g=D_if$, we have
$$
\E_\eta[c_i(g_m-g)^2]
\leq\cE(f_m-f)\longrightarrow0.
$$
Similar to the argument above, applying the second identity of Lemma \ref{lem:identities} to
$f_m-f_\ell$ gives
$$
\E_\eta[c_i|\nabla g_m-\nabla g_\ell|^2]
\leq\E_\eta\left[|A(f_m-f_\ell)|^2\right]
\longrightarrow0.
$$
Thus $\sqrt{c_i}\nabla g_m$ has a limit in $L^2(\eta)$. In particular if $K$ is a compact set contained in the interior of an orthant, $c_i$ is bounded from below on $K$ and so $\nabla g_m\mid_K \longrightarrow \nabla g\mid_K$ in $L^2(\eta)$. Exhausting each open orthant, we identify the $L^2(\eta)$
limit of $\sqrt{c_i}\nabla g_m$ with $\sqrt{c_i}\nabla g$,
$\eta$-almost everywhere. Hence
$$
\E_\eta\left[c_i|\nabla g_m-\nabla g|^2\right]
\longrightarrow0.
$$
This establishes the claim about the integrability of $g$ and its
gradient.

To pass to the limit in Lemma \ref{lem:identities}, recall that
$\nabla^2W$ is bounded and
$$
0\le-\frac{c_j'}{x_j}\leq\frac1{S^2},
\qquad
0\leq 2c_i-x_ic_i'\leq4c_i,
\qquad
0\le\frac{\partial_iW}{x_i}\le K.
$$
Together with the convergences above, these bounds show
that every term in the corresponding identity converges.
This proves that identity for $f$. Finally, in the odd case each $g_m$ is smooth, compactly
supported and unconditional. Applying Lemma~\ref{lem:weighted} to $g_m$
and passing to the limit gives
$$
\left(\lambda-\frac{3}{2S^2}\right)
\left(
\E_\eta[c_i g^2]
-\frac{(\E_\eta[c_i g])^2}{\E_\eta[c_i]}
\right)
\le\E_\eta[c_i|\nabla g|^2].
$$
\end{proof}

\bibliographystyle{alpha}
\bibliography{bib}

\end{document}